\documentclass[11pt]{article}
\usepackage{amssymb}
\usepackage{amsmath}
\input{epsf.sty} 
\usepackage{color}
\newtheorem{theorem}{Theorem}[section] 
\newtheorem{corollary}{Corollary}[section]

\newtheorem{remark}{Remark}[section]
\newtheorem{proposition}{Proposition}[section]
\newcommand{\E}{{\rm E}}
\newcommand{\Prob}{{\rm P}}
\newenvironment{proof}{\begin{trivlist}
\item[\hspace{\labelsep}{\bf\noindent Proof. }]}
{$\hfill\Box$\end{trivlist}}

\newcommand{\barF}{\overline{F}}
\title{\huge\bf A generalized telegraph process\\
 with velocity driven by random trials\footnote{First published in 
 {\em Advances in Applied Probability}, Vol.\  45, No.\ 4, p.\ 1111-1136 \textcircled{c} 2013 by The Applied Probability Trust}
 }
\author{Irene Crimaldi\footnote{IMT Institute for Advanced Studies,  
Piazza San Ponziano 6, I-55100 Lucca, Italy, 
{irene.crimaldi@imtlucca.it}},
{Antonio  Di Crescenzo}, 
{Antonella Iuliano}, 
{Barbara Martinucci}\footnote{Dipartimento di Matematica, Universit\`a di Salerno,
Via Ponte don Melillo, Fisciano (SA), I-84084, Italy,
\{adicrescenzo, aiuliano, bmartinucci\}@unisa.it}
}
\begin{document}
%----------------------------------------------------------------------------
\maketitle
        
\begin{abstract}
We consider a random trial-based telegraph process, which describes a motion on 
the real line with two constant velocities along opposite directions. At each epoch of 
the underlying counting process the new velocity is determined by the outcome of a 
random trial. Two schemes are taken into account: Bernoulli trials and classical P\'olya 
urn trials. We investigate the probability law of the process and the mean of the velocity 
of the moving particle. We finally discuss two cases of interest: (i) the case of Bernoulli 
trials and intertimes having exponential distributions with linear rates (in which, interestingly, 
the process exhibits a logistic stationary density with non-zero mean), and (ii) the case of 
P\'olya trials and intertimes having first Gamma and then exponential distributions with 
constant rates.
  

 
\medskip\noindent
{\bf Running title} \ Telegraph process driven by random trials.

\end{abstract}
%
%=========================================================================
\section{Introduction} 
%========================================================================
Since the 1950s several authors investigated the (integrated) telegraph 
process as a realistic model of random motion (see Goldstein
\cite{Go51}, Bartlett \cite{Ba57} and Kac \cite{Kac74}).  Such process 
describes a motion of a particle on the real line characterized by constant 
speed, the direction being reversed at the random epochs of a Poisson
process.  The probability density of the particle's position satisfies
a hyperbolic differential equation, whose probabilistic properties
have been studied for instance by Orsingher \cite{Ors1990},
\cite{Ors1995}, Foong and Kanno \cite{FoKa94}, and more recently by
Beghin {\em et al.\/} \cite{BeNiOr01}. See also the book of Pinsky \cite{Pi1991}, 
in which many results on the telegraph process and its generalizations  
were given, and many applications were discussed. Certain one-dimensional 
generalizations of the telegraph process focus on cases in which the 
intertimes between two consecutive changes of direction are characterized 
by a variety of distributions.  
We recall the case of Erlang distribution (Di Crescenzo \cite{DiCr2001}), 
of gamma distribution (Di Crescenzo and Martinucci \cite{DiCrMa2007}), 
of the exponential distribution with linearly increasing rate 
(Di Crescenzo and Martinucci \cite{DiCrMa2010}). Moreover, a generalized telegraph 
process governed by an alternating renewal process was studied by Zacks \cite{Za2004}, 
whereas Iacus \cite{Ia2001} gave a rare example where an explicit probability law is 
obtained for an inhomogeneous telegraph process.  
\par
In this paper we aim to study the telegraph model subject to a further source of 
randomness, by assuming that the velocity of the moving particle is driven by random 
trials. Precisely, we deal with a two-velocity random motion on the real line where, 
differently from the classical telegraph process whose positive and negative 
velocities are alternating, at each time epoch now the new velocity is determined 
by the outcome of a random trial. The latter follows one out of two schemes: 
the {\em Bernoulli scheme\/}, which acts with independence, and the classical 
{\em P\'olya scheme\/} (cf.\ P\'olya \cite{Po31}, Mahmoud \cite{Ma2008}), 
where the outcome of each trial depends on the outcomes of the previous ones.   
We remark that the inclusion of random trials in finite-velocity random evolutions  
allows to describe some real situations of interest, such as the motions of particles 
subject to collisions, whose effects may produce direction changes. 
\par
Some novelties with respect to various finite-velocity random motions are 
mentioned in Section 2, with special reference to models with random velocities. 
We recall that investigations on the telegraph process with random velocities 
have been performed recently by Stadje and Zacks \cite{StZa2004}, and by 
De Gregorio \cite{DeGr2010}. Moreover, some recent contributions on multidimensional 
motions characterized by finite speed and randomly distributed directions are given in 
De Gregorio \cite{DeGr2012}, and De Gregorio and Orsingher \cite{DeGrOrs2012}. 
\par
The paper is organized as follows. In Section \ref{model} we describe in detail the 
mathematical model of the motion. In Section \ref{gen-results}, we investigate the
probability law and the conditional mean velocity of the process. Then, in Section 
\ref{bernoulli-case}, for the Bernoulli scheme we discuss the instance in which the 
random intertimes between consecutive trials are exponentially distributed with 
linearly increasing rates. In this case we obtain the probability density of the process 
in closed form, we show that it possesses a logistic stationary density, and then express 
the conditional mean velocity of the process in terms of hypergeometric functions.  
Finally, in Section \ref{polya-case}, we show a case for the P\'olya scheme in which 
the first random intertimes in both directions are Gamma-distributed, whereas all 
remaining intertimes are exponentially distributed. We obtain the probability density 
in closed form and the conditional mean velocity as series of Gauss hypergeometric
functions. For the reader's convenience, the paper is enriched by two Appendices 
containing definitions and formulas used in the proofs.
%==========================================================================
\section{Stochastic model}\label{model}
%===========================================================================
Let $\left\{(S_{t},V_{t});t\geq0\right\}$ be a continuous-time
stochastic process, where $S_{t}$ and $V_{t}$ denote respectively
position and velocity at time $t$ of the moving particle.  The motion
is characterized by two velocities, $c$ and $-v$, with $c,v>0$, the
direction of the motion being specified by the sign of the
velocity. At time $T_{0}=0$ the particle starts from the origin, thus
$S_{0}=0$. The initial velocity $V_{0}$ is determined by the outcome
$X_1$ of the first random trial. At the random time $T_{1}>0$ the
particle is subject to an event, whose effect possibly changes the
velocity according to the outcome $X_2$ of the second random trial.  This
behavior is repeated cyclically at every instant of a sequence of
random epochs $T_0=0<T_{1}<T_{2}<T_{3}<\cdots$. We assume that the
durations of time intervals $[T_{n},T_{n+1})$, $n=0,1,2,\ldots$, 
constitute a sequence of non-negative random variables. Precisely, 
let $U_{k}$ (resp.\ $D_{k}$) denote the random duration of the 
$k$-th time period during which the particle moves forward, with 
velocity $c$ (resp.\ backward, with velocity $-v$). Furthermore, 
$\{U_{k};k\geq 1\}$ and $\{D_{k};k\geq 1\}$ are mutually independent 
sequences of non-negative and absolutely continuous independent 
random variables. Denoting by $Z_{n}$ the velocity of the particle 
during the interval $[T_{n},T_{n+1})$, we assume that $\{Z_{n}; n\geq 0\}$ 
is a sequence of random variables governed by the sequence 
$\{X_n; n\geq 1\}$ of the outcomes of the random trials. Moreover we 
assume that the collection $\{U_k,D_k; n\geq 1\}$ is independent 
of $\{X_n;n\geq 1\}$.
\par In this paper we focus on the case in which the random trials
$\{X_n; n\geq 1\}$ are the outcomes of a sequence of indicator
functions such that
\begin{equation}\label{eq:V0}
  \Prob\{Z_{0}=c\}=\Prob\{X_1=1\}=\frac{b}{b+r}, 
\qquad \Prob\{Z_{0}=-v\}=\Prob\{X_1=0\}=\frac{r}{b+r}  
\end{equation}
and for $n\geq 1$
\begin{equation}\label{eq:Vn}
\begin{split}
\Prob\{Z_{n}=c\,|\,{\cal G}_n\}&=\Prob\{X_{n+1}=1\,|\,{\cal G}_n\}=
\frac{b+A\sum_{k=1}^nX_k}{b+r+An},\\ 
\Prob\{Z_{n}=-v\,|\,{\cal G}_n\}&=\Prob\{X_{n+1}=0\,|\,{\cal G}_n\}=
\frac{r+A\sum_{k=1}^n(1-X_k)}{b+r+An}, 
 \end{split}
\end{equation}
where $b$ and $r$ are positive constants, $A$ is a non-negative
constant, ${\cal G}_0=\{\emptyset,\Omega\}$ and ${\cal
  G}_n=\sigma(X_1,\dots,X_n)$ for $n\geq 1$.
\par We can distinguish two different cases.
\begin{itemize}
\item[(i)] The case $A=0$, which means that the random trials
  $\{X_n;n\geq 1\}$ are independent, i.e. they are a {\em Bernoulli scheme}
  with parameter 
\begin{equation}
 p=\frac{b}{b+r}.
 \label{eq:def_p}
\end{equation}
\item[(ii)] The case $A>0$, which means that $\{X_n;n\geq 1\}$ are the
  outcomes of a sequence of drawings from an urn which initially has
  $b$ black balls and $r$ red balls and then is updated according to
  the {\em classical P\'olya urn scheme}: the drawn ball is returned
  into the urn together with $A>0$ balls of the same
  color.\footnote{We recall that, from a mathematical point of
      view, also in this case the parameters $b,r,A$ of the model can
    be {\em real} (not necessarily integer).}  
($X_1=1$ means that the drawn ball is black; otherwise we have $X_1=0$.) 
In this case the random trials $\{X_n;n\geq 1\}$ are not independent, but only exchangeable 
(see Aldous \cite{Al85}). We recall that urn schemes are used in many applications in order to 
model the so-called {\em preferential attachment principle}, which is a key feature governing the 
dynamics of many economic, social and biological systems. It can be formulated as follows: 
the greater the number of times we observed a certain event, the higher is the probability 
of occurrence of that event at next time. 
\par In Section \ref{sect:reinforced} we discuss an extension of the
above setting to the case in which $\{X_n; n\geq 1\}$ is the sequence
of the outcomes of drawings from a {\em randomly reinforced urn}.
\end{itemize}
We notice that Stadje and Zacks \cite{StZa2004} studied a telegraph process with random velocities, 
where at each epoch of a homogeneous Poisson process the new velocity of the motion is chosen 
according to a common density, independently of the previous velocities and of the Poisson process. 
Hence, in particular the velocities are a sequence of i.i.d.\ random variables and, in Section 5 
of \cite{StZa2004}, the first-exit time of the process through a positive constant  is investigated 
in the special case of 2-valued random velocities. This schema corresponds to the case (i) 
described above, where at each random epoch a Bernoulli trial occurs. However, there is a  
significant difference: whereas in \cite{StZa2004} the random durations are independent of the 
random velocities, our model is based on more general assumptions, which involve non-identically 
distributed durations of random intervals $[T_n,T_{n+1})$, depending on the values taken by the 
velocities. Moreover our model includes the case (ii) of not independent random velocities.  
These two facts are also novelties with respect to De Gregorio \cite{DeGr2010} where, as 
in \cite{StZa2004}, a homogeneous Poisson process governs the velocity changes that form 
a sequence of i.i.d.\ random variables, independent of the Poisson process. 
\par 
Let $M_t$ be the stochastic process which counts the number
of epochs $T_{i}$ (with $i\geq 1$) occurring before $t$, i.e.
\begin{equation} 
 M_t=\max\left\{i \geq 1: T_{i}\leq t\right\},
 \qquad t>0.
 \label{eq:defMt}
\end{equation}
It is worthwhile to note that position and velocity of the particle at
time $t$ can thus be formally expressed as:
\begin{eqnarray}
 V_{t}=Z_{M_t}, 
 \qquad 
 S_{t}={\int^{t}_{0}}\,V_{s}\,{\rm d}s, 
 \qquad t>0.
 \label{eq:VP}
\end{eqnarray}
Figure 1 shows two examples of sample paths of $S_t$, with indication
of the random intertimes $U_k$ and $D_k$, where the sequence 
$\{X_n; n\geq 1\}$ takes values {\em (a)} $\{1,1,0,0,1,0,1,\ldots\}$ 
and {\em (b)} $\{0,1,1,0,1,1,\ldots\}$, respectively.  
%
%::::::::::::::::::::::::::::::::::::::::::::::::::::::::::::::::::::::::::::
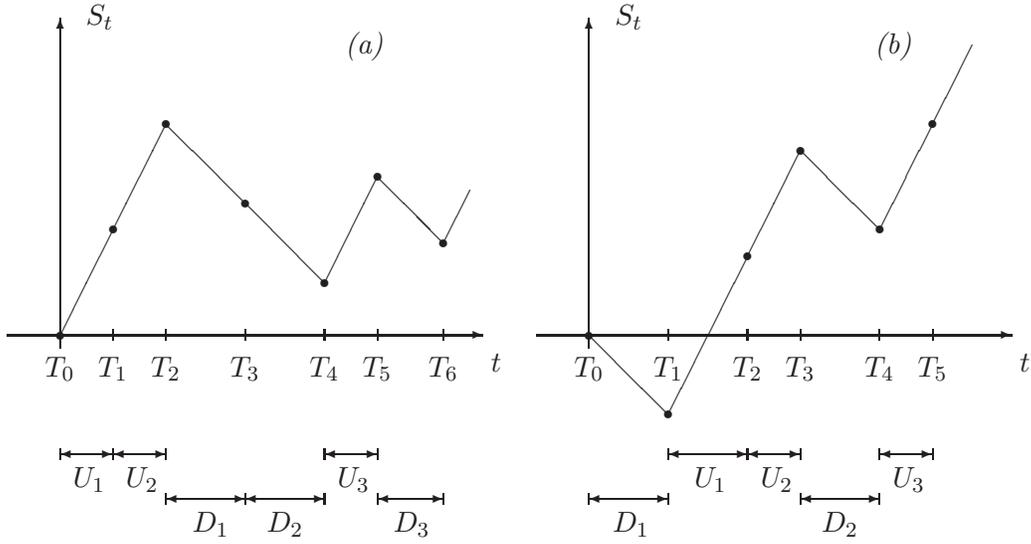
\begin{figure}[t]
\begin{center}
\begin{picture}(381,176) 
\put(20,160){\makebox(30,15)[t]{$S_t$}} 
\put(220,160){\makebox(30,15)[t]{$S_t$}} 
\put(120,150){\makebox(30,15)[t]{{\em (a)}}} 
\put(320,150){\makebox(30,15)[t]{{\em (b)}}} 
\put(0,50){\vector(1,0){180}} 
\put(200,50){\vector(1,0){180}} 
\put(20,45){\vector(0,1){125}} 
\put(220,45){\vector(0,1){125}} 
\put(175,28){\makebox(20,15)[t]{$t$}} 
\put(375,28){\makebox(20,15)[t]{$t$}} 
\put(20,50){\line(1,2){40}} 
\put(60,130){\line(1,-1){60}} 
\put(120,70){\line(1,2){20}} 
\put(140,110){\line(1,-1){25}}
\put(165,85){\line(1,2){10}}
\put(220,50){\line(1,-1){30}} 
\put(250,20){\line(1,2){50}} 
\put(300,120){\line(1,-1){30}} 
\put(330,90){\line(1,2){35}} 
\put(40,48){\line(0,1){4}} 
\put(60,48){\line(0,1){4}} 
\put(90,48){\line(0,1){4}} 
\put(120,48){\line(0,1){4}} 
\put(140,48){\line(0,1){4}} 
\put(165,48){\line(0,1){4}} 
\put(250,48){\line(0,1){4}} 
\put(280,48){\line(0,1){4}} 
\put(300,48){\line(0,1){4}} 
\put(330,48){\line(0,1){4}} 
\put(350,48){\line(0,1){4}} 
\put(20,50){\circle*{3}} 
\put(40,90){\circle*{3}} 
\put(60,130){\circle*{3}} 
\put(90,100){\circle*{3}} 
\put(120,70){\circle*{3}} 
\put(140,110){\circle*{3}} 
\put(165,85){\circle*{3}} 
\put(220,50){\circle*{3}} 
\put(250,20){\circle*{3}} 
\put(280,80){\circle*{3}} 
\put(300,120){\circle*{3}} 
\put(330,90){\circle*{3}} 
\put(350,130){\circle*{3}} 
\put(5,27){\makebox(30,15)[t]{$T_0$}} 
\put(25,27){\makebox(30,15)[t]{$T_1$}} 
\put(45,27){\makebox(30,15)[t]{$T_2$}} 
\put(75,27){\makebox(30,15)[t]{$T_3$}} 
\put(105,27){\makebox(30,15)[t]{$T_4$}} 
\put(125,27){\makebox(30,15)[t]{$T_5$}} 
\put(150,27){\makebox(30,15)[t]{$T_6$}} 
\put(205,27){\makebox(30,15)[t]{$T_0$}} 
\put(235,27){\makebox(30,15)[t]{$T_1$}} 
\put(265,27){\makebox(30,15)[t]{$T_2$}} 
\put(285,27){\makebox(30,15)[t]{$T_3$}} 
\put(315,27){\makebox(30,15)[t]{$T_4$}} 
\put(335,27){\makebox(30,15)[t]{$T_5$}} 
\put(20,7){\line(0,-1){4}} 
\put(40,7){\line(0,-1){4}} 
\put(60,7){\line(0,-1){4}} 
\put(60,-10){\line(0,-1){4}} 
\put(90,-10){\line(0,-1){4}} 
\put(120,-10){\line(0,-1){4}} 
\put(120,7){\line(0,-1){4}} 
\put(140,7){\line(0,-1){4}} 
\put(140,-10){\line(0,-1){4}} 
\put(165,-10){\line(0,-1){4}} 
\put(220,-10){\line(0,-1){4}} 
\put(250,-10){\line(0,-1){4}} 
\put(250,7){\line(0,-1){4}} 
\put(280,7){\line(0,-1){4}} 
\put(300,7){\line(0,-1){4}} 
\put(300,-10){\line(0,-1){4}} 
\put(330,-10){\line(0,-1){4}} 
\put(330,7){\line(0,-1){4}} 
\put(350,7){\line(0,-1){4}} 
\put(20,5){\vector(1,0){20}} 
\put(40,5){\vector(-1,0){20}} 
\put(40,5){\vector(1,0){20}} 
\put(60,5){\vector(-1,0){20}} 
\put(60,-12){\vector(1,0){30}} 
\put(95,-12){\vector(-1,0){35}} 
\put(90,-12){\vector(1,0){30}} 
\put(115,-12){\vector(-1,0){25}} 
\put(120,5){\vector(1,0){20}} 
\put(140,5){\vector(-1,0){20}} 
\put(140,-12){\vector(1,0){25}} 
\put(165,-12){\vector(-1,0){25}} 
\put(220,-12){\vector(1,0){30}} 
\put(250,-12){\vector(-1,0){30}} 
\put(250,5){\vector(1,0){30}} 
\put(280,5){\vector(-1,0){30}} 
\put(280,5){\vector(1,0){20}} 
\put(300,5){\vector(-1,0){20}} 
\put(300,-12){\vector(1,0){30}} 
\put(330,-12){\vector(-1,0){30}} 
\put(330,5){\vector(1,0){20}} 
\put(350,5){\vector(-1,0){20}} 
\put(16,-15){\makebox(30,15)[t]{$U_1$}} 
\put(36,-15){\makebox(30,15)[t]{$U_2$}} 
\put(62,-32){\makebox(30,15)[t]{$D_1$}} 
\put(90,-32){\makebox(30,15)[t]{$D_2$}} 
\put(116,-15){\makebox(30,15)[t]{$U_3$}} 
\put(138,-32){\makebox(30,15)[t]{$D_3$}} 
\put(221,-32){\makebox(30,15)[t]{$D_1$}} 
\put(251,-15){\makebox(30,15)[t]{$U_1$}} 
\put(276,-15){\makebox(30,15)[t]{$U_2$}} 
\put(300,-32){\makebox(30,15)[t]{$D_2$}} 
\put(326,-15){\makebox(30,15)[t]{$U_3$}} 
\end{picture} 
\end{center} 
\vspace{0.6cm}
\caption{Sample paths of $S_t$ with {\em (a)} $V_{0}=c$, and {\em (b)}
  $V_{0}=-v$.}
\end{figure}
%::::::::::::::::::::::::::::::::::::::::::::::::::::::::::::::::::::::::::::
% 
\par 
The following stochastic equation holds:
$$
 S_{T_{k+1}}=S_{T_k}+W_k, \qquad k\geq 0\,,
$$ 
where $\{W_k; k\geq 0\}$ is the sequence of random variables defined by
$$
W_0=\left\{
 \begin{array}{ll}
 c U_1 & \hbox{if }Z_0=c\\
 -v D_1 & \hbox{if }Z_0=-v,
 \end{array}
\right.
$$
and, for $k\geq 1$ and $1\leq j\leq k+1$,
$$
 W_k=\left\{
 \begin{array}{ll}
 c U_j & \hbox{if }Z_k=c  \hbox{ and } X_1+N_{k-1}=j-1\\
 -v D_j & \hbox{if }Z_k=-v \hbox{ and } X_1+N_{k-1}=k-j+1\,,
 \end{array}
 \right.
$$
where $N_{k-1}$ is the random variable that counts the number of random trials 
yielding velocity $c$ among the trials going from the $2$-nd to the $k$-th one, i.e.
\begin{equation} 
 N_{k-1}=\sum_{i = 1}^{k-1}\textbf{1}_{\left\{Z_{i}=c\right\}}
 =\sum_{i=1}^{k-1}X_{i+1} 
 =\sum_{h=2}^k X_h\,,
 \qquad k\geq 2. 
 \label{number-draw}
\end{equation}
For convenience, we set also $N_0=0$.  In Appendix \ref{appendix-N} we
give some useful formulas regarding the probability distribution 
of  $N_{k-1}$ conditioned on the initial velocity for the two different schemes.  
\par 
Throughout the paper we denote by $f_{U_{k}}$ (resp.\ $F_{U_{k}}$,
${\overline{F}}_{U_{k}}$) and by $f_{D_{k}}$ (resp.\ $F_{D_{k}}$,
${\overline{F}}_{D_{k}}$) the probability densities (resp.\ cumulative
distribution functions, tail distribution functions) of $U_{k}$ and $D_{k}$,
respectively.  Moreover, $f^{(k)}_{U}$ and $f^{(k)}_{D}$ (resp., $F^{(k)}_{U}$ and $F^{(k)}_{D}$, 
$\barF^{(k)}_{U}$ and $\barF^{(k)}_{D}$) will denote the probability densities
(resp., cumulative distribution functions, tail distribution functions) of the partial sums
$$
 U^{(k)}=U_{1}+U_{2}+\cdots+U_{k},\qquad D^{(k)}=D_{1}+D_{2}+\cdots+D_{k}, 
 \qquad k\geq 1\,.
$$
%===========================================================================
\section{Probability law and mean velocity}\label{gen-results}
%===========================================================================
At time $t=0$ we assume $S_0=0$ so that at time $t>0$ the particle is
located in the domain $[-vt,ct]$.  The probability law of $S_{t}$,
$t>0$, thus possesses a discrete component on the points $-vt$ and $ct$, 
and an absolutely continuous component over $(-vt,ct)$, which
will be investigated in the sequel. 
\begin{proposition}\label{prop:discr}
For all $t>0$ we have
\begin{equation} \label{Prob1}
\Prob\{S_{t}=ct\} 
 =\frac{b}{b+r}\overline{F}_{U_{1}}(t) + \sum_{k = 1}^{+\infty}
\Prob\{Z_0=c,N_{k}=k\}\left[F_U^{(k)}(t)-F_U^{(k+1)}(t)\right],
\end{equation}
where 
$$
 \Prob\{Z_0=c,N_{k}=k\}=
 \left\{
 \begin{array}{ll}
 p^{k+1} & \hbox{if $A=0$} \\
\displaystyle 
\frac{b}{b+r}\left(\frac{b+A}{A}\right)_k  
\left[\left(\frac{b+A+r}{A}\right)_k\right]^{-1} &  
\hbox{if  $A>0$}.  
\end{array}
\right.
$$
Similarly, we have 
\begin{equation}\label{Prob1-neg}
\Prob\{S_{t}=-vt\}
=\frac{r}{b+r}\overline{F}_{D_{1}}(t)+ \sum_{k = 1}^{+\infty}
 \Prob\{Z_0=-v,N_{k}=0\}
\left[F_D^{(k)}(t)-F_D^{(k+1)}(t)\right],
\end{equation}
where 
$$
 \Prob\{Z_0=-v,N_{k}=0\}=
 \left\{
 \begin{array}{ll}
 (1-p)^{k+1} & \hbox{if $A=0$} \\
\displaystyle 
\frac{r}{b+r}\left(\frac{r+A}{A}\right)_k  
\left[\left(\frac{b+A+r}{A}\right)_k\right]^{-1} &  
\hbox{if  $A>0$}.  
\end{array}
\right.
$$ 
\end{proposition}
\begin{proof} First of all, note that, for $y\in\{-v,c\}$, condition  
$S_{t}=yt$ implies $V_0=y$ and $V_{t}=y$ a.s.
%  ho scritto almost surely perch\'e, ad esempio, si ha $S_t=ct$ 
%anche quando la 
%  velocit\`a \`e sempre $c$ fino a $t$ escluso (ossia su $[0,t[$) e 
%vale $t=T_{k}$ 
%  e $V_t=-v$. Per\`o l'evento $T_k=t$ ha probabilit\'a zero}}
Conditioning on $M_t=k$ and recalling (\ref{number-draw}) we have
\begin{equation*}
\begin{split}
\Prob\{S_t=ct\}&=\Prob\{S_{t}=ct, V_{t}=c, V_{0}=c\}\\
&=\sum_{k=0}^{+\infty}\Prob\{Z_0=c,N_k=k\}\,\Prob\{M_{t}=k\,|\,Z_0=c,N_k=k\}.
\end{split}
\end{equation*}
Eq.\ ({\ref{Prob1}}) thus easily follows from (\ref{eq:V0}) and from
Eqs.\ (\ref{formula-bern2}) and (\ref{formula-polya2}) for $j=k-1$. In
the same way Eq.\ ({\ref{Prob1-neg}}) follows from (\ref{eq:V0}) and
from Eqs.\ (\ref{formula-bern3-neg}) and (\ref{formula-polya3-neg}).
\end{proof}
\par
Let us now define the
probability density of $S_t$, $t>0$, conditional on initial velocity 
$y \in \left\{-v,c\right\}$:
\begin{equation}
 p(x,t\,|\,y)=\frac{\partial}{\partial x}
 \Prob\{S_{t}\leq x\,|\,V_{0}=y\}, 
 \qquad x \in (-vt,ct). 
 \label{defdensp}
\end{equation}
We remark that, for $t>0$, 
$$
\Prob\{S_{t}=yt\,|\,V_0=y\}+
\int_{-vt}^{ct}p(x,t\,|\,y)\,{\rm d}x=1, 
 \qquad y\in\{-v,c\}.
$$
The density of the particle position is
\begin{equation}
 p(x,t)=\frac{\partial}{\partial x}\Prob\{S_{t}\leq x\}
 =\frac{b}{b+r}\,p(x,t\,|\,c)+\frac{r}{b+r}\,p(x,t\,|\,-v),
 \label{eq:defdensp}
\end{equation}
where $p(x,t\,|\,y)$, defined in (\ref{defdensp}), can be expressed as 
\begin{equation}
 p(x,t\,|\,y)
 =f(x,t\,|\,y)+b(x,t\,|\,y).
 \label{densp}
\end{equation}
Here $f$ and $b$ denote the densities of particle's position when the
motion at time $t$ is characterized by forward and backward velocity,
respectively, i.e.\ 
\begin{eqnarray*}
&& \hspace*{-0.8cm} 
f(x,t\,|\,y)=\frac{\partial}{\partial x}
 \Prob\{S_{t}\leq x, V_{t}=c\,|\, V_{0}=y\}, 
\\
&& \hspace*{-0.8cm}
b(x,t\,|\,y)=
\frac{\partial}{\partial x}
 \Prob\{S_{t}\leq x, V_{t}=-v\,|\, V_{0}=y\}, 
\end{eqnarray*}
for $x$ in $(-vt,ct)$. Then, for $y \in \left\{-v,c\right\}$, we can write
\begin{equation}
 f(x,t\,|\,y)=\sum_{k=1}^{+\infty}f_k(x,t\,|\,y), 
 \qquad 
 b(x,t\,|\,y)=\sum_{k=1}^{+\infty}b_k(x,t\,|\,y)
 \label{density-sum}
\end{equation}
where, for $x$ in $(-vt,ct)$,
\begin{eqnarray}
&& \hspace*{-0.8cm} 
f_k(x,t\,|\,y)=\frac{\partial}{\partial x}
\Prob\{S_{t}\leq x, V_{t}=c, M_t=k\,|\, V_{0}=y\}, 
\label{densf_k}
\\
&& \hspace*{-0.8cm}
b_k(x,t\,|\,y)=
\frac{\partial}{\partial x}
\Prob\{S_{t}\leq x, V_{t}=-v, M_t=k\,|\, V_{0}=y\},
\label{densb_k}
\end{eqnarray}
with $M_t$ defined in (\ref{eq:defMt}). 
It is worthwhile to recall that, in the case of a telegraph process driven by i.i.d.\ random velocities 
and an independent homogeneous Poisson process, a two-dimensional renewal equation for 
$p(x,t)$ is provided in Eq.\ (2.5) of \cite{StZa2004}. However, as pointed out by the authors, 
it is quite difficult to be solved analytically. In our case hereafter we develop a different approach, 
based on suitable conditioning. We first give an expression for densities $f_k$ and $b_k$ 
conditioned by $V_{0}=c$. 
Note that $f_1(x,t\,|\,c)=0$.
\begin{theorem}\label{Theor1}
For $t>0$ and $-vt < x < ct$, the densities (\ref{densf_k}) and
(\ref{densb_k}) can be expressed as
\begin{equation}
\begin{split}
 f_k(x,t\,|\,c) &= \frac{1}{c+v} \sum_{j = 0}^{k-2}
\Prob\{N_{k-1}=j,Z_k=c\,|\,Z_0=c\}
\\
&\times  f^{(k-j-1)}_{D}(t-\tau_{*})
{\int^{t}_{t-\tau_{*}}}f^{(j+1)}_{U}(s-t+\tau_{*})
\overline{F}_{U_{j+2}}(t-s){\rm d}s,
\qquad k\geq 2,
\end{split}
\label{Prob2}
\end{equation}
and
\begin{equation}
\begin{split}
 b_k(x,t\,|\,c)& =\frac{1}{c+v}\Big\{ 
\Prob\{N_{k-1}=k-1,Z_k=-v\,|\,Z_0=c\}
f^{(k)}_{U}(\tau_{*})\overline{F}_{D_1}(t-\tau_{*}) 
\\
& +{\bf 1}_{\{k \geq 2\}} \sum_{j = 0}^{k-2}
\Prob\{N_{k-1}=j,Z_k=-v\,|\,Z_0=c\} 
\\
& \times f^{(j+1)}_{U}(\tau_{*})
{\int^{t}_{\tau_{*}}}f^{(k-j-1)}_{D}(s-\tau_{*})
\overline{F}_{D_{k-j}}(t-s){\rm d}s\Big\},
\qquad k\geq 1,
\end{split}
\label{Prob3}
\end{equation}
where
\begin{equation}
 \tau_{*}=\tau_{*}(x,t)=\frac{vt+x}{c+v},
 \label{tau}
\end{equation}
and where $\Prob\{N_{k-1}=j,Z_k=c\,|\,Z_0=c\}$ and 
$\Prob\{N_{k-1}=j,Z_k=-v\,|\,Z_0=c\}$ are given respectively by
(\ref{formula-bern2}) and (\ref{formula-bern3}) if $A=0$, and by 
(\ref{formula-polya2}) and (\ref{formula-polya3}) if $A>0$.
\end{theorem}
%%%
\begin{proof}
Recalling (\ref{densf_k}), for $t > 0$, $-vt < x < ct$, and $k\geq 2$ we have 
\begin{eqnarray*}
&& \hspace*{-0.7cm}
f_k(x,t\,|\,c){\rm d}x 
={\int^{t}_{0}}
\Prob\{T_{k}\in {\rm d}s,Z_{k}=c, S_{s}+c(t-s)\in {\rm d}x, T_{k+1}-T_{k}>t-s, 
\\
&& \hspace*{3.2cm}
0\leq N_{k-1} \leq k-2\,|\, Z_{0}=c\}.
\end{eqnarray*}
(Case $k=1$ does not give an absolutely continuous component). 
Conditioning on $N_{k-1}$, and taking into account the number of time
periods during which the particle moved forward and backward, we obtain
\begin{eqnarray*}
&& \hspace*{-0.7cm}
f_k(x,t\,|\,c){\rm d}x 
=\sum_{j = 0}^{k-2}
 {\int^{t}_{0}}
\Prob\{
U^{(j+1)}+D^{(k-j-1)}\in{\rm d}s,cU^{(j+1)}-vD^{(k-j-1)}+c(t-s)\in{\rm d}x\} 
\nonumber
\\
&& \hspace*{1.3cm}
\times \Prob\{U_{j+2}>t-s\}\Prob\{N_{k-1}=j,Z_k=c\,|\,Z_{0}=c\}.
\nonumber
\end{eqnarray*}
Note that conditions $S_{s}+c(t-s)=x$ and $S_{s}\geq -vs$ provide 
$s\geq (ct-x)/(c+v)=t- \tau_{*}$. 
This inequality and the independence of $U^{(k)}$ and $D^{(k)}$ thus give 
$$
 f_k(x,t\,|\,c)=\sum_{j = 0}^{k-2}\Prob\{N_{k-1}=j,Z_k=c\,|\,Z_0=c\}
 {\int^{t}_{t- \tau_{*}}} h(s,x-c(t-s))\Prob\{U_{j+2}>t-s\}{\rm d}s,
$$ 
where $h(\cdot,\cdot)$ is the joint probability density of
$(U^{(j+1)}+D^{(k-j-1)},cU^{(j+1)}-vD^{(k-j-1)})$.  Since
$$
h(s,x-c(t-s))=
\frac{1}{c+v}{f^{(j+1)}_{U}\left(s-\frac{ct-x}{c+v}\right)f^{(k-j-1)}_{D}
\left(\frac{ct-x}{c+v}\right)},
$$ Eq.\ (\ref{Prob2}) thus follows recalling
(\ref{tau}). Eq.\ (\ref{Prob3}) can be obtained in a similar
way. Indeed for $k\geq 1$
\begin{eqnarray*}
&& \hspace*{-0.7cm}
b_k(x,t\,|\,c){\rm d}x 
={\int^{t}_{0}}
\Prob\{T_{k}\in {\rm d}s,Z_{k}=-v, S_{s}-v(t-s)\in {\rm d}x, T_{k+1}-T_{k}>t-s, 
\\
&& \hspace*{2.8cm}
0\leq N_{k-1} \leq k-1\,|\, Z_{0}=c\}.
\end{eqnarray*}
Conditioning on $N_{k-1}$, and taking into account the number of time
periods during which the particle moved forward and backward, we
obtain for $k\geq 1$
\begin{equation}
\begin{split}
b_k(x,t\,|\,c){\rm d}x &=
{\int^{t}_{0}}\Prob\{U^{(k)}\in{\rm d}s,cU^{(k)}-v(t-s)\in{\rm d}x\} 
\\
& \times \Prob\{D_1>t-s\}\Prob\{N_{k-1}=k-1,Z_k=-v\,|\,Z_{0}=c\}
\\
& +\sum_{j = 0}^{k-2}
 {\int^{t}_{0}}\Prob\{
U^{(j+1)}+D^{(k-j-1)}\in{\rm d}s,cU^{(j+1)}-vD^{(k-j-1)}-v(t-s)\in{\rm d}x\} 
\\
& \times \Prob\{D_{k-j}>t-s\}\Prob\{N_{k-1}=j,Z_k=-v\,|\,Z_{0}=c\}.
\end{split}
\nonumber
\end{equation}
(Note that for $k=1$ the above term
$\Prob\{N_{k-1}=k-1,Z_k=-v\,|\,Z_{0}=c\}$ reduces to
$\Prob\{Z_1=-v\,|\,Z_{0}=c\}$ and the sum on $j$ is equal to zero.)
For the first integral, conditions $U^{(k)}=s$ and $cU^{(k)}-v(t-s)=x$
imply $s=\tau_*$; while in the second integral, conditions
$S_{s}-v(t-s)=x$ and $S_{s}\leq cs$ provide $s\geq
(x+vt)/(c+v)=\tau_{*}$.  This inequality and the independence of
$U^{(k)}$ and $D^{(k)}$ thus give for the above sum on $j$
$$
\sum_{j = 0}^{k-2}
\Prob\{N_{k-1}=j,Z_k=-v\,|\,Z_0=c\}
{\int^{t}_{\tau_{*}}}
 h(s,x+v(t-s))\Prob\{D_{k-j}>t-s\}{\rm d}s,
$$ where, as above, $h(\cdot,\cdot)$ is the joint probability density
 of $(U^{(j+1)}+D^{(k-j-1)},cU^{(j+1)}-vD^{(k-j-1)})$.  Since
$$
h(s,x+v(t-s))=
\frac{1}{c+v}{f^{(j+1)}_{U}\left(\tau_*\right)f^{(k-j-1)}_{D}
\left(s-\tau_*\right)}\,,
$$
we get (\ref{Prob3}).
\end{proof}
\begin{remark}\label{remark-density}
\rm Similarly to (\ref{Prob2}) and (\ref{Prob3}), when the
initial velocity is negative, for all $t>0$ and $-vt < x < ct$, the
densities (\ref{densf_k}) and (\ref{densb_k}) are expressed as
\begin{equation}
\begin{split}
 f_k(x,t\,|\,-v) &=\frac{1}{c+v}\Big\{\Prob\{N_{k-1}=0,Z_k=c\,|\,Z_0=-v\}
f^{(k)}_{D}(t-\tau_{*})\overline{F}_{U_1}(\tau_{*})
\\
&+{\bf 1}_{\{k\geq 2\}} \sum_{j = 0}^{k-2}
\Prob\{N_{k-1}=k-1-j,Z_k=c\,|\,Z_0=-v\}
\\
& \times f^{(j+1)}_{D}(t-\tau_{*})
 \int_{t-\tau_{*}}^{t} f^{(k-j-1)}_{U}(s-t+\tau_{*})
 \overline{F}_{U_{k-j}}(t-s){\rm d}s\Big\},
 \qquad k\geq 1,
\end{split}
\nonumber 
\end{equation}
\begin{equation}
\begin{split}
b_k(x,t\,|\,-v) &=\frac{1}{c+v} \sum_{j = 0}^{k-2}
\Prob\{N_{k-1}=k-1-j,Z_k=-v\,|\,Z_0=-v\}
\\
&\times f^{(k-j-1)}_{U}(\tau_{*}) \int^{t}_{\tau_{*}}f^{(j+1)}_{D}(s-\tau_{*})
\overline{F}_{D_{j+2}}(t-s){\rm d}s,
\qquad k\geq 2,
\end{split}
\nonumber
\end{equation}
where $\tau_{*}$ is defined in (\ref{tau})  and where 
$\Prob\{N_{k-1}=k-1-j,Z_k=c\,|\,Z_0=-v\}$ and 
$\Prob\{N_{k-1}=k-1-j,Z_k=-v\,|\,Z_0=-v\}$ are given respectively by 
(\ref{formula-bern2-neg}) and (\ref{formula-bern3-neg}) if $A=0$, 
and by (\ref{formula-polya2-neg}) and (\ref{formula-polya3-neg}) if $A>0$.
\end{remark}
\par 
In the following proposition, conditional on positive initial velocity, the mean 
of process $V_{t}$ is expressed in terms of the cumulative distribution 
function of $T_k$.  (Obviously, we can obtain a similar expression conditioning 
on negative initial velocity.) We remark that, by conditioning on the value of 
$N_{k-1}$ and using the independence of $U^{(j+1)}$ and $D^{(k-j-1)}$, it follows:
\begin{equation}
 F_{T_k|Z_0}(t\,|\,c)=\sum_{j=0}^{k-1}\Prob\{N_{k-1}=j\,|\,Z_0=c\}
 \Prob\{U^{(j+1)}+D^{(k-j-1)}\leq t\},
 \label{FTk|Z0}
\end{equation}
where 
\begin{equation}
 \Prob\{U^{(j+1)}+D^{(k-j-1)}\leq t\}
 =\int_0^t F_D^{(k-j-1)}(t-s)f_U^{(j+1)}(s)\,{\rm{d}}s
 =\int_0^t F_U^{(j+1)}(t-s)f_D^{(k-j-1)}(s)\,{\rm{d}}s.
 \label{FsumUD}
\end{equation}
\begin{proposition}\label{th-mediaV}
For all $t > 0$ we have
\begin{equation}
\E [V_{t}\,|\,V_{0}=c]=c\,\overline{F}_{U_{1}}(t)+ 
c\pi_A\, \sum_{k = 1}^{+\infty}\phi_k(t\,|\,c)
+\,(-v)(1-\pi_A)\,\sum_{k = 1}^{+\infty}\psi_k(t\,|\,c)
 \label{mediaV}
\end{equation}
where $\pi_A=\displaystyle\frac{b+A}{b+A+r}$ and 
\begin{equation}
\begin{split}
\phi_k(t\,|\,c)&=
F_{T_k|Z_0}(t\,|\,c)- {\int^{t}_{0}}F_{T_{k}|Z_0}(t-s\,|\,c)f_{U_{k+1}}(s){\rm{d}}s, 
\\
\psi_k(t\,|\,c)&= 
F_{T_k|Z_0}(t\,|\,c)-{\int^{t}_{0}}F_{T_{k}|Z_0}(t-s\,|\,c)f_{D_{k+1}}(s){\rm{d}}s.
\end{split}
\label{phi-gen}
\end{equation}
\end{proposition}
\begin{proof}
We recall that in both cases (by independence in the Bernoulli scheme
and by exchangeability in P\'olya scheme), we have for each $n\geq 2$
\begin{equation}
\Prob\{X_{n}=1\,|\,X_1=1\}=\pi_A,
\qquad 
\Prob\{X_{n}=0\,|\,X_1=1\}=1-\pi_A\,. 
\label{legge-X}
\end{equation}
Hence, for every positive integer $k\geq 1$, recalling
(\ref{eq:defMt}) and the first of (\ref{eq:VP}), we have
$$
 \E [V_{t}\,|\,V_{0}=c]= \E [Z_{M_t}\,|\,Z_{0}=c]
 =c\overline{F}_{U_1}(t)+
\sum_{k=1}^{+\infty}\E[Z_{k}\cdot \textbf{1}_{\left\{T_{k}\leq t < T_{k+1}\right\}}\,|\,Z_0=c].
$$
Since, for $k\geq 1$,
\begin{equation*}
\begin{split}
\E[Z_{k}\cdot \textbf{1}_{\left\{T_{k}\leq t < T_{k+1}\right\}}\,&|\,Z_0=c]\\
&
=\,c\,\Prob\{Z_k=c\,|\,Z_0=c\}\,
\E[\textbf{1}_{\left\{T_{k}\leq t < T_{k+1}\right\}}\,|\,Z_0=c,Z_{k}=c]\\
&+\,
(-v)\,\Prob\{Z_k=-v\,|\,Z_0=c\}\,
\E[\textbf{1}_{\left\{T_{k}\leq t < T_{k+1}\right\}}\,|\,Z_0=c,Z_{k}=-v]
\\
&
=\,c\,\Prob\{X_{k+1}=1\,|\,X_1=1\}\,
\Prob\{T_{k} \leq t < T_{k+1}\,|\,Z_0=c,Z_{k}=c\}\\
&+\,
(-v)\Prob\{X_{k+1}=0\,|\,X_1=1\}\,\Prob\{T_{k} \leq t < T_{k+1}\,|\,Z_0=c,Z_{k}=-v\}
\\
&
=\,c\pi_A\,
\int^{\infty}_{0}\Prob(t-s <T_{k} \leq t\,|\,Z_0=c) f_{U_{k+1}}(s){\rm{d}}s
\\
&+\,(-v)(1-\pi_A)\,
\int^{\infty}_{0}\Prob(t-s <T_{k} \leq t\,|\,Z_0=c) f_{D_{k+1}}(s){\rm{d}}s
\\
& 
=\,c\pi_A\, \phi_k(t\,|\,c)+\,(-v)(1-\pi_A)\,\psi_k(t\,|\,c)\,.
\end{split}
\end{equation*}
Eq.\ (\ref{mediaV}) thus follows, due to (\ref{phi-gen}).
\end{proof}
\par 
The following remark turns out to be useful in some particular cases.
\begin{remark}\label{remarkFT}
\rm It is worthwhile to note that, using (\ref{FTk|Z0}) and (\ref{FsumUD}), 
Eqs.\ (\ref{phi-gen}) can be rewritten as
\begin{equation*}%\label{phi}
\begin{split}
 \phi_k(t\,|\,c)&=\sum_{j = 0}^{k-1}\Prob\{N_{k-1}=j|Z_0=c\}
 \int_0^t F_D^{(k-j-1)}(t-y)\left[f^{(j+1)}_U(y)-f_{U^{(j+1)}+U_{k+1}}(y)\right]{\rm d}y,\\
 \psi_k(t\,|\,c)&=\sum_{j = 0}^{k-1}\Prob\{N_{k-1}=j|Z_0=c\}
 \int_0^t F_U^{(j+1)}(t-y)\left[f^{(k-j-1)}_D(y)-f_{D^{(k-j-1)}+D_{k+1}}(y)\right]{\rm d}y.
\end{split}
\end{equation*}
\par
For instance, the above quantities can be easily computed when the
random variables $U_k$ (and $D_k$) are gamma distributed with 
the same scale parameter, since it is well-known that, by independence, 
the sums of the involved random variables are still gamma distributed.
\end{remark} 
\begin{remark}
\rm In the special case in which $(Z_k)_k$ and $(T_{k+1}-T_k)_k$ are independent, we have $U_1$ and $D_1$ identically distributed and 
$$
 \phi_k(t\,|\,c)=\psi_k(t\,|\,c)=\Prob\{T_k\leq t<T_{k+1}\}
$$ 
so that we may write Eq.\ (\ref{mediaV}) as
$$
 \E[V_t\,|\,V_0=c]=c\overline{F}_{T_1}(t)+\E[Z_1\,|\,Z_0=c]F_{T_1}(t)\,.
$$ 
Similarly, we have
$$
 \E[V_t\,|\,V_0=-v]=-v\overline{F}_{T_1}(t)+\E[Z_1\,|\,Z_0=-v]F_{T_1}(t)
$$ 
and so we find 
\begin{equation*}
\begin{split}
\E[V_t]&=\left(c\Prob\{Z_0=c\}-v\Prob\{Z_0=-v\}\right)
\overline{F}_{T_1}(t)\\
&+
\left(\E[Z_1\,|\,Z_0=c]\Prob\{Z_0=c\}+\E[Z_1\,|\,Z_0=-v]\Prob\{Z_0=-v\}\right)
F_{T_1}(t)\\
&=\E[Z_0]\overline{F}_{T_1}(t)+\E[Z_1]F_{T_1}(t)\\
&=\E[Z_0]\,.
\end{split}
\end{equation*}
The last equality is due to the fact that the random variables $X_k$, and so $Z_k$, 
are identically distributed in both cases $A=0$ and $A\neq 0$. As a consequence, we get
$$
 \E[S_t]=t\E[Z_0]\,.
$$
This is the same formula found in Stadje and Zacks \cite{StZa2004} for the telegraph process 
driven by i.i.d.\ random velocities and an independent homogeneous Poisson process. 
\end{remark}
\par In the following sections we discuss some special cases arising
in the two different schemes of Bernoulli and P\'olya trials, and
leading to closed forms for the probability law of $S_t$.
%===========================================================================
\section{Particular case for Bernoulli scheme $(A=0)$}
\label{bernoulli-case}
%===========================================================================
The classical telegraph process is characterized by exponentially
distributed times separating consecutive velocity changes. 
The extension of such a model to the case of velocities driven by Bernoulli 
trials can be performed in a simple and natural way. Indeed, if we consider a 
simple telegraph process $\widetilde {S_t}$ with alternating velocities $c$ and $-v$, 
and with alternating switching intensities $\widetilde \lambda:=(1-p)\lambda$ 
and $\widetilde \mu:=p\mu$, then it can be proved that the marginal distributions 
of $\widetilde S_t$, $t\geq 0$, and $S_t$, $t\geq 0$, are identical (where $S_t$ 
is a telegraph process with switching intensities $\lambda$ and $\mu$, and 
having the same velocities of $\widetilde S_t$, governed by Bernoulli trials with 
parameter $p$). Hence, in this case results on $S_t$ can be immediately obtained 
by those of $\widetilde S_t$. Other results on this case, when $\lambda=\mu$, 
can be found in Stadje and Zacks \cite{StZa2004}. 
\par
Therefore, aiming to discuss a non trivial case, and stimulated by previous 
studies (see Di Crescenzo and Martinucci \cite{DiCrMa2010} and 
Di Crescenzo {\em et al.} \cite{DiCr2011}) involving finite-velocity random 
motions with stochastically decreasing random intertimes, in the following we 
assume that the r.v.'s $U_{k}$ and $D_{k}$ have exponential distribution with
linear rates $\lambda k$ and $\mu k$. Hence, the tail distribution functions are
\begin{equation}
 \overline{F}_{U_{k}}(t)={\rm e}^{-\lambda k t},
 \qquad \overline{F}_{D_{k}}(t)={\rm e}^{-\mu k t},
 \qquad t\geq 0,
 \label{distrnew1}
\end{equation}
with $\lambda,\mu >0$. Figure 2 shows some simulations of $S_{t}$ in
the present case, where the particle exhibits a kind of damped motion. 
This special case belongs to a more general framework in which counting processes 
with increasing intensity function are employed in applied fields. A typical example in this 
respect is the non-homogeneous Poisson process with increasing intensity function, 
which deserves interest in reliability contexts involving repairable systems (see, for instance, 
Cohen and Sackrowitz \cite{ChSa93}, or Di Crescenzo and Martinucci \cite{DiCrMa2009} for a 
power-law process). In other cases, as in the present model, the effect of an increasing intensity 
function can be obtained also by assuming increasing arrival rates (see Brown \cite{Br2011}). 
Another example of random motion with shrinking steps is the two-dimensional Pearson walk 
studied in Serino and Redner \cite{SeRe10}. In the latter paper the step size decreases 
deterministically with a geometric rule, whereas in the present model the step length decreases 
stochastically, according to the tail distribution functions specified in (\ref{distrnew1}). 
\par
Due to assumption (\ref{distrnew1}), $U^{(k)}$ and $D^{(k)}$,
$k\geq1$, have generalized exponential densities
\begin{eqnarray}
 f^{(k)}_{U}(t)=k(1-{\rm e}^{-\lambda t})^{k-1}\lambda {\rm e}^{-\lambda t},
 \qquad f^{(k)}_{D}(t)=k(1-{\rm e}^{-\mu t})^{k-1}\mu {\rm e}^{-\mu t},
 \qquad t>0
 \label{dens2}
\end{eqnarray}
with corresponding cumulative distribution functions 
\begin{eqnarray}
F^{(k)}_{U}(t)=(1-{\rm e}^{-\lambda t})^{k},
 \qquad 
F^{(k)}_{D}(t)=(1-{\rm e}^{-\mu t})^{k},
 \qquad t\geq 0.
 \label{cum-distrib2}
\end{eqnarray}
Hence, $U^{(k)}$ and $D^{(k)}$ are distributed as the maximum of $k$
i.i.d.\ random variables having exponential distributions with rates
$\lambda$ and $\mu$, respectively.  
\par
%  ----------------------- figura 2 -----------------------------------
\begin{figure}[t]
\centerline{
\epsfxsize=8cm
\epsfbox{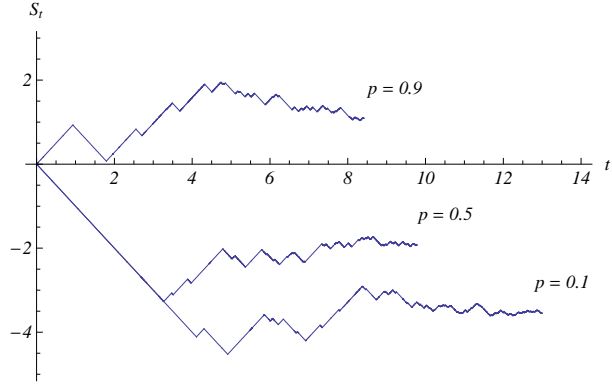}
}
% --------------------------------------------------------------------
\caption{Simulated sample paths of $S_t$ in the Bernoulli scheme, 
exponential damped case, with $\lambda=\mu=1$ and $c=v=1$, 
for some choices of $p$.}
\end{figure}
\par 
By making use of the results in the previous section, in the following 
we obtain the probability law of $S_{t}$. We start providing the discrete
component.
\begin{proposition}\label{prop:discresp}
Let $U_{k}$ and $D_{k}$ be exponentially distributed with rates $\lambda k$ 
and $\mu k$, $k=1,2,\cdots$, respectively. For all $t>0$ we have  
$$
 \Prob\{S_{t}=ct\}=\frac{p\, {e}^{-\lambda t}}
{1 - p \left(1-{e}^{-\lambda t}\right)}\,,
 \qquad
 \Prob\{S_{t}=-vt\}=\frac{(1-p)\,{e}^{-\mu t}}
{1 -(1- p) \left(1-{e}^{-\mu t}\right)}\,.
$$
\end{proposition}
\begin{proof} 
It immediately follows from Proposition \ref{prop:discr} and
Eq.\ (\ref{cum-distrib2}).
\end{proof}
The following theorem and corollary give the absolutely continuous component of
the probability law of $S_t$.
\begin{theorem}\label{th:densdamp}
Under the assumptions of Proposition \ref{prop:discresp}, 
for all $t>0$ and  $-vt < x < ct$, we have  
\begin{eqnarray}
&& \hspace*{-0.4cm}
f(x,t\,|\,c)=
\frac{ (1-p)p\mu {\rm e}^{\mu(t+\tau_{*})} ({\rm e}^{\lambda \tau_{*}}-1)}
{(c+v) \left[ p {\rm e}^{\mu t}+(1-p){\rm e}^{ (\lambda+\mu)\tau_* } \right]^2 }
\label{densitynew1}
\\
&& \hspace*{-0.4cm}
b(x,t\,|\,c)=
\frac{ (1-p)\left[ \lambda p {\rm e}^{ \lambda \tau_*+\mu(t+\tau_*) }+
\lambda (1-p) {\rm e}^{ \lambda\tau_*+ 2\mu\tau_*} \right] }
{ (c+v) \left[ p {\rm e}^{\mu t}+ (1-p) {\rm e}^{ (\lambda+\mu)\tau_*} \right]^2 }\,,
\label{densitynew2}
\end{eqnarray}
where $\tau_{*}$ is defined in $(\ref{tau})$.
\end{theorem}
%%% 
\begin{proof}
Since $U_{1}$ and $U_{k+1}$ are exponentially distributed with parameters 
$\lambda$ and $\lambda(k+1)$, respectively, from (\ref{density-sum}), 
(\ref{Prob2}) and (\ref{formula-bern2}), recalling (\ref{distrnew1}) and (\ref{dens2}) we get
\begin{equation*}
\begin{split}
f(x,t\,|\,c)&=
\frac{\mu}{c+v}\sum_{k = 2}^{+\infty}\sum_{j = 0}^{k-2}
{k-1 \choose j}p^{j+1}(1-p)^{k-1-j}(k-j-1)
{\rm e}^{-\mu(t-\tau_{*})}\left[1-{\rm e}^{-\mu(t-\tau_{*})}\right]^{k-2-j}
\\
&\times\,
{\int^{t}_{t-\tau_{*}}}\lambda(j+1){\rm e}^{-\lambda(s-t+\tau_{*})}
\left[1-{\rm e}^{-\lambda(s-t+\tau_{*})}\right]^{j}{\rm e}^{-\lambda(t-s)(j+2)}{\rm d}s
\\
&= 
\frac{\mu p (1-p)}{c+v}{\rm e}^{-\mu(t-\tau_{*})}
\sum_{k=2}^{+\infty}(k-1)\sum_{j=0}^{k-2}
{k-2 \choose j}p^{j}(1-p)^{k-2-j}\left[1-{\rm e}^{-\mu(t-\tau_{*})}\right]^{k-2-j}
\\
& 
\times {\rm e}^{-\lambda\tau_*}
{\int^{t}_{t-\tau_{*}}}\lambda(j+1){\rm e}^{-\lambda(t-s)}
\left[{\rm e}^{-\lambda(t-s)}-{\rm e}^{-\lambda\tau_{*}}\right]^{j}{\rm d}s
\\
&=
\frac{\mu p (1-p)}{c+v}{\rm e}^{-\mu(t-\tau_{*})-2\lambda \tau_{*}}({\rm e}^{\lambda \tau_{*}}-1)
\\
&\times 
\sum_{k = 2}^{+\infty}(k-1)\sum_{j = 0}^{k-2}{k-2 \choose j}
\left[p(1-{\rm e}^{-\lambda \tau_{*}})\right]^{j}
\left[(1-p)\left(1-{\rm e}^{-\mu(t-\tau_{*})}\right)\right]^{k-j-2}
\\
&=\frac{\mu p(1-p)}{c+v}{\rm e}^{-\mu(t-\tau_{*})-2\lambda \tau_{*}}({\rm e}^{\lambda \tau_{*}}-1)
 \,\frac{{\rm e}^{2\mu t+2\lambda \tau_{*}}}{\left[(1-p){\rm e}^{(\lambda+\mu)\tau_{*}}
+p\,{\rm e}^{\mu t}\right]^{2}}.
\end{split}
\end{equation*}
This yields density (\ref{densitynew1}). Eq.\ (\ref{densitynew2}) can
be obtained from (\ref{density-sum}), (\ref{Prob3}) and
(\ref{formula-bern3}) in a similar way.  Indeed, for $k=1$ we have
$$
b_1(x,t\,|\,c)=\frac{\lambda(1-p)}{c+v}{\rm e}^{-\lambda\tau_*-\mu(t-\tau_*)}
$$
and, for $k\geq 2$,
\begin{equation*}
\begin{split}
b_k(x,t\,|\,c)&=\frac{1}{c+v}
\left\{
p^{k-1}(1-p)k\left(1-{\rm e}^{-\lambda\tau_*}\right)^{k-1}\lambda {\rm e}^{-\lambda\tau_*-\mu(t-\tau_*)}
+\right.
\\
&+
\left. 
\sum_{j=0}^{k-2}{k-1 \choose j}p^j(1-p)^{k-j}(j+1)\left(1-{\rm e}^{-\lambda\tau_*}\right)^j
\lambda {\rm e}^{-\lambda\tau_*} \right.\\
&
\left. \times \int_{\tau_*}^t(k-j-1)\left(1-{\rm e}^{-\mu(s-\tau_*)}\right)^{k-j-2}
\mu {\rm e}^{-\mu(s-\tau_*)-\mu(k-j)(t-s)}\,{\rm d}s
\right\}
\\
&=
 \frac{\lambda}{c+v}
\left\{ {\rm e}^{-\lambda\tau_*-\mu(t-\tau_*)}
(1-p)k\left[p(1-{\rm e}^{-\lambda\tau_*})\right]^{k-1} 
\right.
\\
&\left. +
{\rm e}^{-\lambda\tau_*}\sum_{j=0}^{k-2}{k-1 \choose j}\left[p(1-{\rm e}^{-\lambda\tau_*})\right]^j
(1-p)^{k-j}(j+1)\right.\\
&
\left. \times {\rm e}^{-\mu(t-\tau_*)}\int_{\tau_*}^t(k-j-1)
\left[{\rm e}^{-\mu(t-s)}-{\rm e}^{-\mu(t-\tau_*)}\right]^{k-j-2}
\mu {\rm e}^{-\mu(t-s)}\,{\rm d}s
\right\}\\
&=
 \frac{\lambda {\rm e}^{-\lambda\tau_*-\mu(t-\tau_*)}(1-p)}{c+v}
\left\{k \alpha^{k-1}+\sum_{j=0}^{k-2}{k-1 \choose j}(j+1)\alpha^j
\beta^{k-1-j}\right\},
\end{split}
\end{equation*}
where we have set 
$$
 \alpha=\alpha(x,t)=p\left(1-{\rm e}^{-\lambda\tau_*}\right)
 \qquad\mbox{and}\qquad 
 \beta=\beta(x,t)=(1-p)\left(1-{\rm e}^{-\mu(t-\tau_*)}\right).
$$
Hence, recalling that (due to binomial theorem)  
$$
\sum_{j=0}^{k-2}{k-1 \choose j}(j+1)\alpha^j \beta^{k-1-j}=
(\alpha+\beta)^{k-1}+(k-1)\alpha(\alpha+\beta)^{k-2}-k\alpha^{k-1}\,,
$$
from the second of (\ref{density-sum}) we finally get     
\begin{equation*}
\begin{split}
b(x,t\,|\,c)&= 
\frac{\lambda {\rm e}^{-\lambda\tau_*-\mu(t-\tau_*)}(1-p)}{c+v}
\left\{1+\sum_{k=2}^{+\infty}(\alpha+\beta)^{k-1}+
\alpha\sum_{k=2}^{+\infty}(k-1)(\alpha+\beta)^{k-2}
\right\}
\\
&=\frac{\lambda {\rm e}^{-\lambda\tau_*-\mu(t-\tau_*)}(1-p)}{c+v}
\frac{1-\beta}{[1-(\alpha+\beta)]^2}\,,
\end{split}
\end{equation*}
which coincides with Eq.\ (\ref{densitynew2}).
\end{proof}
\par
We are now able to find out the probability density of $S_t$ in closed form.
\begin{corollary}\label{th:problaw1}
%\label{rem:pdatoc}
Under the assumptions of Proposition \ref{prop:discresp}, 
for all $t>0$ and $-vt < x < ct$, we have 
\begin{equation}
 p(x,t\,|\,c)=\frac{1-p}{c+v}
\,
\frac{ \lambda(1-p){\rm e}^{(\lambda + 2 \mu)\tau_{*}}
 +\,p\,{\rm e}^{\mu(t+\tau_{*})}\left[(\lambda +\mu){\rm e}^{\lambda \tau_{*}}-\mu\right]}
{ \left[p {\rm e}^{\mu t}+(1-p){\rm e}^{(\lambda+\mu)\tau_*}\right]^2 }
 \label{eq:psmorzc}
\end{equation}
and 
\begin{equation}
 p(x,t\,|\,-v)=
\frac{p}{c+v}
\,
\frac{ {\rm e}^{\mu(t+\tau_*)}\left[(\lambda+\mu) {\rm e}^{\lambda\tau_*}(1-p) +\mu p\right]
-\lambda (1-p) {\rm e}^{(\lambda+2\mu)\tau_*} }
{ \left[p {\rm e}^{\mu t}+(1-p){\rm e}^{(\lambda+\mu)\tau_*}\right]^2 },
 \label{eq:psmorzv}
\end{equation}
where $\tau_{*}$ is defined in $(\ref{tau})$. Hence, for $-vt < x < ct$ we have
\begin{equation}
 p(x,t)=\frac{p}{1-p} \,\frac{1}{s} \,
 \frac{\exp\left\{(\mu- \frac{v}{s})t-\frac{1}{s}x\right\}}
 {\left[1+\frac{p}{1-p} \,\exp\left\{(\mu-\frac{v}{s})t-\frac{1}{s}x\right\}\right]^2},
 \label{densnewp}
\end{equation} 
where $s:=(c+v)/(\lambda+\mu)$. 
\end{corollary}
\begin{proof}
Recalling (\ref{densp}), from densities (\ref{densitynew1}) and (\ref{densitynew2}) 
we obtain (\ref{eq:psmorzc}). By symmetry, density (\ref{eq:psmorzv}) can be 
expressed from (\ref{eq:psmorzc}) by replacing $\tau_{*}$ by $(t-\tau_{*})$ and 
interchanging $\lambda$ with $\mu$ and $p$ with $(1-p)$. Therefore we finally get 
$$ 
p(x,t)=\frac{
p(1-p)(\lambda+\mu)\,{e}^{\mu t+(\lambda+\mu)\tau_{*}}} 
{(c+v)[(1-p) \,{e}^{(\lambda+\mu)\tau_{*}}+p{e}^{\mu t}]^2},
\qquad -vt < x < ct,
$$ 
and then Eq.\ (\ref{densnewp}) easily follows.
\end{proof}
\par
Some plots of density (\ref{densnewp}) are given in Figures 3, 4 and 5 
for various choices of the involved parameters. 
\begin{remark}
\rm 
We can analyse the behavior of the density (\ref{densnewp}) when 
$x$ tends to the endpoints of the state space $[-vt,ct]$. For $t>0$ we have
$$
 \lim_{x \downarrow -vt}p(x,t)
 =\frac{p}{1-p} \,\frac{1}{s} \,
 \frac{{e}^{\mu t}}{\left[1+\frac{p}{1-p} \,{e}^{\mu t}\right]^2}, 
 \qquad 
 \lim_{x \uparrow ct}p(x,t)
 =\frac{p}{1-p} \,\frac{1}{s} \,
 \frac{{e}^{-\lambda t}}{\left[1+\frac{p}{1-p} \,{e}^{-\lambda t}\right]^2}.
$$
\end{remark}
%
%%%%%%%%%  FIG 3  %%%%%% 
\begin{figure}[t]
\centerline{
\epsfxsize=7.cm
\epsfbox{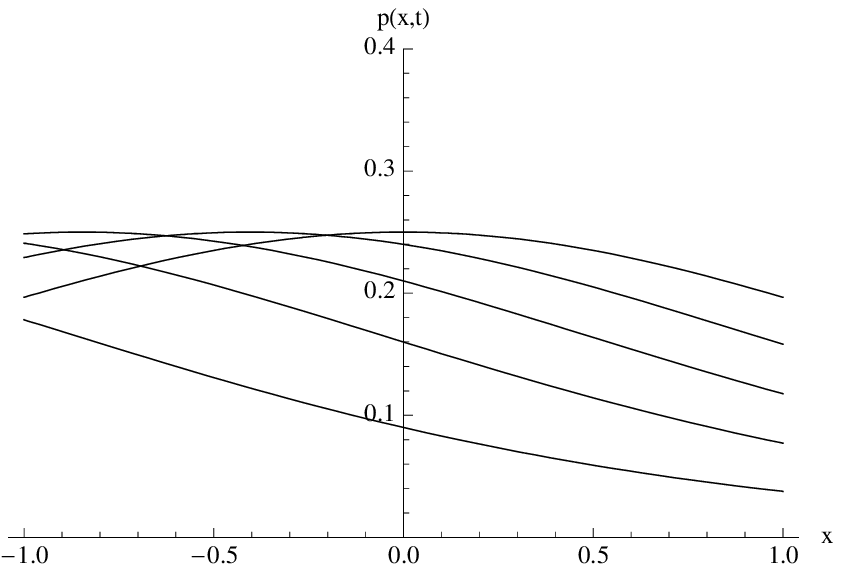}
\;
\epsfxsize=7.cm
\epsfbox{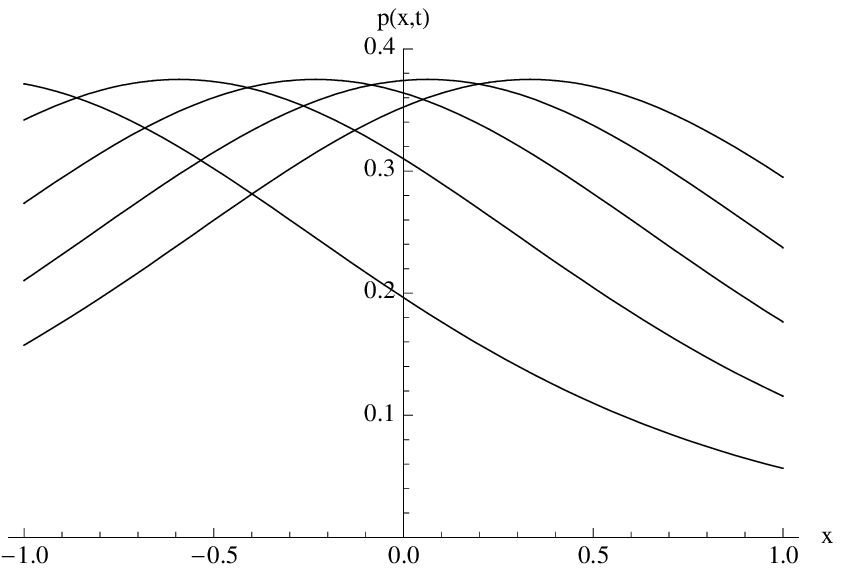}
}
\caption{
Density (\ref{densnewp}), with 
$t=1$, $c=v=1$, $\lambda=1$, for $\mu=1$ ({\em left}) and $\mu=2$ ({\em right}), 
and $p=0.1$, $0.2$, $0.3$, $0.4$, $0.5$ (from bottom to top near $x=1$).
}
\end{figure}
%%%%%%%%%%%%%%%% 
%
%%%%%%%%%  FIG 4  %%%%%% 
\begin{figure}[t]
\centerline{
\epsfxsize=7.cm
\epsfbox{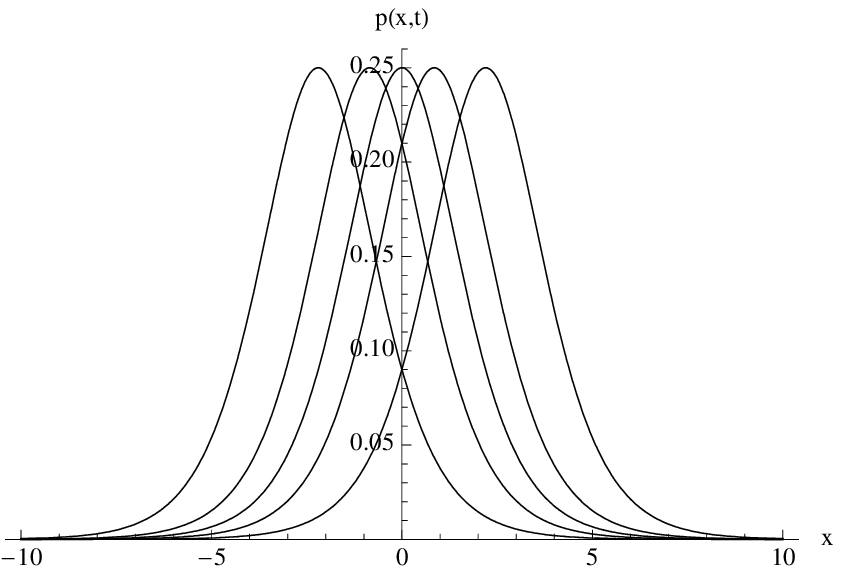}
\;
\epsfxsize=7.cm
\epsfbox{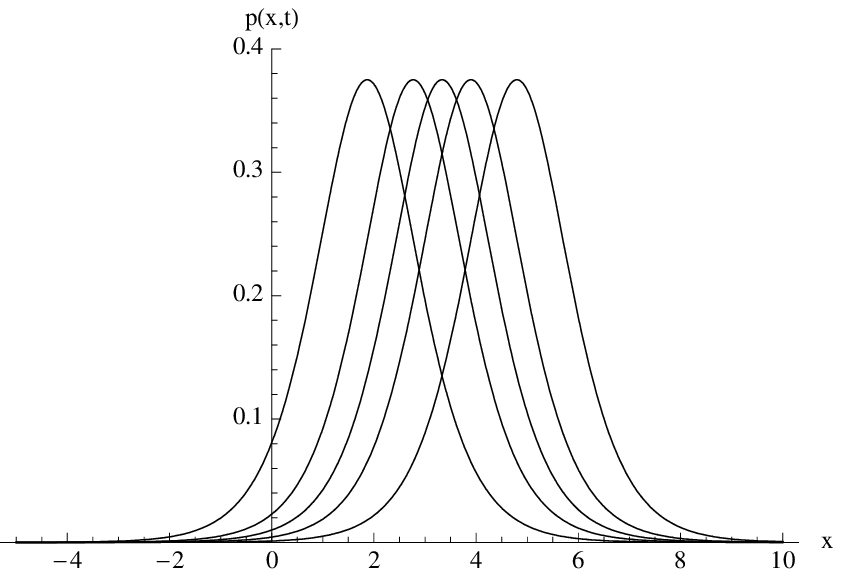}
}
\caption{Same as Figure 3, with $t=10$ and $p=0.1$, $0.3$, $0.5$, $0.7$, $0.9$
  (from left to right in both plots).}
\end{figure}
%%%%%%%%%%%%%%%% 
%
%%%%%%%%%  FIG 5  %%%%%% 
\begin{figure}[t]
\centerline{
\epsfxsize=7.cm
\epsfbox{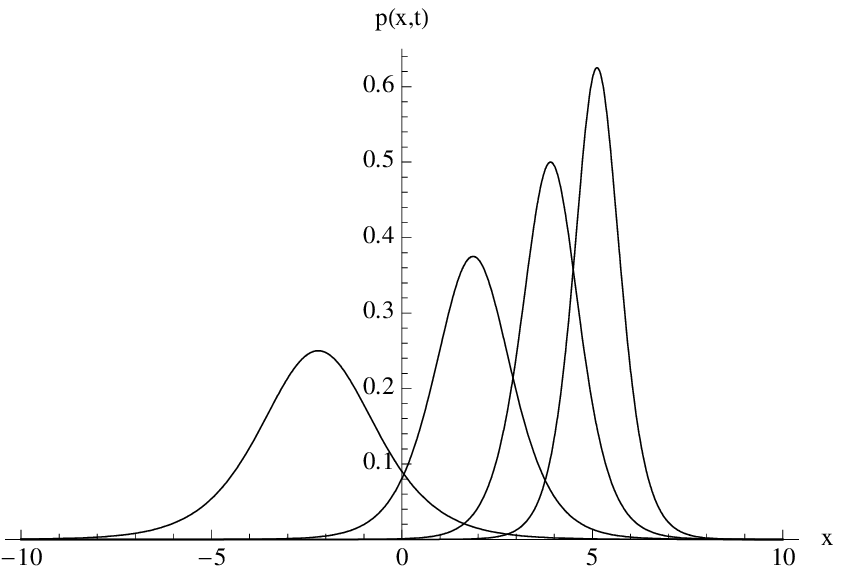}
\;
\epsfxsize=7.cm
\epsfbox{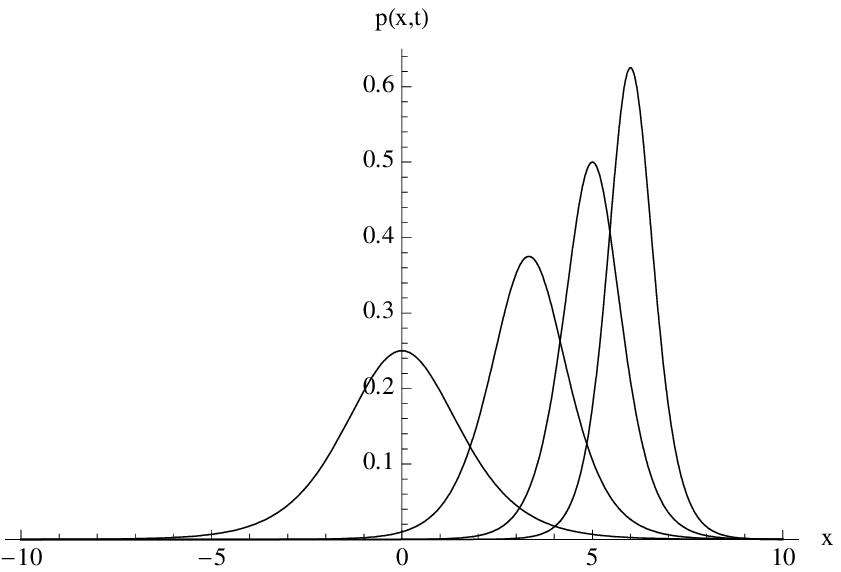}
}
\caption{Same as Figure 4, 
for $p=0.1$ ({\em left}) and $p=0.5$ ({\em right}), 
with  $\mu=1$, $2$, $3$, $4$ (from left to right in both plots).}
\end{figure}
%%%%%%%%%%%%%%
%
\par
The following special case follows by straightforward calculations.
\begin{corollary} 
Under the assumptions of Proposition \ref{prop:discresp}, 
if $\lambda\,v=\mu\,c$ then density (\ref{densnewp}) becomes  the 
truncated logistic density given by 
$$
p(x,t)
 =\frac{{e}^{-(x-m)/s}}{s\left[1+ {e}^{-(x-m)/s}\right]^{2}},
 \qquad -vt < x < ct,\;\; t>0,
$$
$$
  m= s\,\ln\left(\frac{p}{1-p}\right),
  \qquad
  s=\frac{v}{\mu}=\frac{c}{\lambda}.
$$
Hence, in this case $S_t$ admits the stationary density
\begin{equation}
 \lim_{t\to +\infty}p(x,t)
 =\frac{{e}^{-(x-m)/s}}{s\left[1+ {e}^{-(x-m)/s}\right]^{2}},
 \qquad x \in  \mathbb{R}.
\label{staz}
\end{equation}
\end{corollary}
\par 
We note that the right-hand-side of Eq.\ (\ref{staz}) is a
logistic density with mean $m$ and variance $\pi^{2}s^{2}/3$. In
addition, if $p=1/2$ then the mean $m$ vanishes, and the density
identifies with the stationary p.d.f.\ of a damped telegraph process,
as obtained in Corollary 3.3 of Di Crescenzo and Martinucci
\cite{DiCrMa2010}. We remark that if $\lambda\,v\neq\mu\,c$ then
$$
 \lim_{t\to +\infty}p(x,t)=0, \qquad x \in \mathbb{R}.
$$
%%%%%%%%%%%%%%%%%%%%%%%%
\par
Finally, let us obtain the mean velocity conditioned by positive initial
velocity. 
\begin{proposition}\label{th:mediaV-2}
Under the assumptions of Proposition \ref{prop:discresp}, 
for all $t> 0$ we have
$$
\E [V_{t}\,|\,V_{0}=c]
 =c\,{\rm e}^{-\lambda t}+ 
 c\pi_A\, \sum_{k = 1}^{+\infty}\phi_k(t\,|\,c)
 +\,(-v)(1-\pi_A)\,\sum_{k = 1}^{+\infty}\psi_k(t\,|\,c),
$$
where 
\begin{equation*}
\begin{split}
\phi_k(t\,|\,c)&=\lambda \sum_{j = 0}^{k-1}(j+1)
{k-1 \choose j} p^j(1-p)^{k-1-j} 
\sum_{\ell=0}^{k-j-1}\sum_{h=0}^j
{k-j-1 \choose \ell }{j \choose h}(-1)^{\ell+h}
\\
& \times \left[t {\rm e}^{-\mu\ell t} \,{}_1F_1(1,2;[\mu\ell-\lambda(h+1)]t)\right.
\\
&\left. -\lambda(k+1)\,{\rm e}^{-\lambda(h+1) t}\, H(\mu\ell-\lambda(h+1), \lambda(k-h);t)\right],
\\
\psi_k(t\,|\,c)&=\lambda \sum_{j = 0}^{k-1}(j+1){k-1 \choose j} p^j(1-p)^{k-1-j} 
\sum_{\ell=0}^{k-j-1}\sum_{h=0}^j {k-j-1 \choose \ell }{j \choose h}(-1)^{\ell+h}
\\
& \times \left[t{\rm e}^{-\mu\ell t} \, {}_1F_1(1,2;[\mu\ell-\lambda(h+1)]t)\right.
\\
&\left. -\mu(k+1)\,{\rm e}^{-\lambda(h+1) t}\, H(\mu\ell-\lambda(h+1), \mu(k+1)-\lambda(h+1);t)\right],
\end{split}
\end{equation*}
and where the function $H$ is defined by (\ref{funcH}) in Appendix \ref{formulas}.
\end{proposition}
\begin{proof} 
The above formulas easily follow from Theorem \ref{th-mediaV}.
Indeed, after some calculations, from (\ref{FTk|Z0}) and
(\ref{FsumUD}) we find the following cumulative distribution function:
\begin{equation*}
\begin{split}
F_{T_k|Z_0}(t\,|\,c) & = \lambda \sum_{j = 0}^{k-1}(j+1)
{k-1 \choose j} p^j(1-p)^{k-1-j} 
\sum_{\ell=0}^{k-j-1}\sum_{h=0}^j
{k-j-1 \choose \ell }{j \choose h}(-1)^{\ell+h} \\
&\times  t {\rm e}^{-\mu\ell t} \,{}_1F_1(1,2;[\mu\ell-\lambda(h+1)]t),
\end{split}
\end{equation*}
where ${}_1F_1(1,2;0)=1$ and   ${}_1F_1(1,2;z)=({\rm e}^z-1)/z$, for $z\neq 0$. 
Moreover, the thesis follows by computing the two integrals which 
appear in the right-hand-sides of Eqs.\ (\ref{phi-gen}).
\end{proof}
%
%============================================================================
\section{Particular case for P\'olya scheme $(A>0)$}
\label{polya-case}
%============================================================================
In this section we consider a special case in which the velocity changes are governed 
by the P\'olya urn scheme. Let us assume that the distributions of $U_1$ and $D_1$ are
$\Gamma\left(\frac{b}{A}+1,\lambda\right)$ and $\Gamma\left(\frac{r}{A}+1,\mu\right)$, 
respectively, and the intertimes $U_k$ and $D_k$, for $k\geq 2$, are exponential with 
parameters $\lambda$ and $\mu$, respectively.  Therefore, $U^{(k)}$ has Gamma
distribution $\Gamma\left(\frac{b}{A}+k,\lambda\right)$ and $D^{(k)}$ has Gamma 
distribution $\Gamma\left(\frac{r}{A}+k,\mu\right)$, so that  
\begin{equation}
 f^{(k)}_{U}(t)=
 \frac{\lambda^{b/A+k}\,t^{b/A+k-1}\,{\rm e}^{-\lambda t}}{\Gamma\left(\frac{b}{A}+k\right)},
 \qquad f^{(k)}_{D}(t)=\frac{\mu^{r/A+k}\,t^{r/A+k-1}\,{\rm e}^{-\mu t}}
 {\Gamma\left(\frac{r}{A}+k\right)},\qquad t>0.
 \label{distr2}
\end{equation}
Notice that $U_1$ (resp., $D_1$) is stochastically larger than $U_k$ (resp., $D_k$), 
$k\geq 2$. Hence, the first time interval along both directions is stochastically greater than 
the other ones in the same direction. This is not an unusual assumption, since in renewal 
theory the distribution of the first interarrival time is often supposed different from that of 
the other ones (see Chapter 2 of Wolff \cite{Wo89}, for instance), as for the delayed 
renewal processes. Moreover, differently from the case treated in the previous section, 
$U_1$ (resp., $D_1$) depends on $b$ and $A$ (resp., $r$ and $A$), which are the 
same parameters involved in the P\'olya urn scheme described in Section 2. 
\par 
Under the above assumptions we are able to explicitly obtain the probability law of $S_{t}$ 
in terms of the hypergeometric function $_{1}F_{1}(u,v;z)$ (see Eq.\ (\ref{hypergeom-func}) 
in Appendix \ref{formulas}). We first provide the tail distribution functions of $U^{(k)}$ and
$D^{(k)}$ (see Eq.\ (\ref{cum-distr-func-gamma}) in Appendix \ref{formulas}):
\begin{equation}\label{surv-polya-gamma-U}
\overline{F}^{(k)}_{U}(t)=
1-\frac{(\lambda t)^{b/A+k}{\rm e}^{-\lambda t}}{\Gamma\left(\frac{b}{A}+k+1\right)}
 {_{1}}F_{1}\left(1,\frac{b}{A}+k+1;\lambda t\right), \qquad t\geq 0,
\end{equation}
\begin{equation}\label{surv-polya-gamma-D}
\overline{F}^{(k)}_{D}(t)=
1-\frac{(\mu t)^{r/A+k}{\rm e}^{-\mu t}}{\Gamma\left(\frac{r}{A}+k+1\right)}
 {_{1}}F_{1}\left(1,\frac{r}{A}+k+1; \mu t\right),\qquad t\geq 0\,.
\end{equation}
\par
In the following proposition we give the discrete component of the probability law of $S_t$. 
\begin{proposition}\label{polya-discr}
Let $U_1$ and $D_1$ be gamma distributed with parameters 
$\left(\frac{b}{A}+1,\lambda\right)$ and $\left(\frac{r}{A}+1,\mu\right)$, 
respectively, and let $U_k$ and $D_k$, for $k\geq 2$, be exponentially distributed 
with parameters $\lambda$ and $\mu$, respectively. For all $t>0$ we have 
\begin{equation}\label{Prob1-polya}
\Prob\{S_{t}=ct\}=
\frac{b}{b+r}\overline{F}_{U_1}(t)
+
\frac{b\left(\lambda t\right)^{b/A} {\rm e}^{-\lambda t}}
{(b+r)\,\Gamma\left(\frac{b+A}{A}\right)}
\left[_{1}F_{1}\left(1,\frac{b+A+r}{A};\lambda t\right)-1\right]\,,
\end{equation}
where $\overline{F}_{U_1}(t)$ is given by (\ref{surv-polya-gamma-U})
for $k=1$. Similarly,
\begin{equation}\label{Prob1-neg-polya}
\Prob\{S_{t}=-vt\}=
\frac{r}{b+r}\overline{F}_{D_1}(t)
+
\frac{r \left(\mu t\right)^{r/A} {\rm e}^{-\mu t}}
{(b+r)\,\Gamma\left(\frac{r+A}{A}\right)}
\left[_{1}F_{1}\left(1,\frac{b+A+r}{A};\mu t\right)
-1\right]\,,
\end{equation}
where $\overline{F}_{D_1}(t)$ is given by (\ref{surv-polya-gamma-D}) for $k=1$.
\end{proposition}
\begin{proof}
Eq.\ (\ref{Prob1-polya}) follows from (\ref{Prob1}) using (\ref{distr2}). Indeed, since
%  (si potrebbe anche sfruttare l'identit\`a 
%  $F_U^{(k)}(t)-F_U^{(k+1)}(t)= \overline  F_U^{(k+1)}(t)- \overline F_U^{(k)}(t)$ ?)
%
\begin{equation*}
F_U^{(k)}(t)-F_U^{(k+1)}(t)=
{\int^{t}_{0}}f^{(k)}_{U}(s)\overline{F}_{U_{k+1}}(t-s)\,{\rm d}s\,,
\end{equation*}
we obtain 
\begin{equation*}
\begin{split}
\Prob\{S_{t}=ct\}&=
\frac{b}{b+r}\overline{F}_{U_1}(t)
+
\sum_{k = 1}^{+\infty}
\Prob\{Z_0=c, N_{k}=k\} 
{\int^{t}_{0}} \lambda^{b/A+k} 
\frac{s^{b/A+k-1}{\rm e}^{-\lambda s}}{\Gamma\left(\frac{b}{A}+k\right)}
{\rm e}^{-\lambda(t-s)}{\rm d}s\\
&=
\frac{b}{b+r}\overline{F}_{U_1}(t)
+
\frac{b{\rm e}^{-\lambda t}\left(\lambda t\right)^{b/A}}
{(b+r)\Gamma\left(\frac{b+A}{A}\right)}
\sum_{k = 1}^{+\infty}
\frac{(1)_k}{\left(\frac{b+A+r}{A}\right)_k}
\frac{\left(\lambda t\right)^{k}}{k!}
\\
&=
\frac{b}{b+r}\overline{F}_{U_1}(t)
+
\frac{b{\rm e}^{-\lambda t}\left(\lambda t\right)^{b/A}}
{(b+r)\Gamma\left(\frac{b+A}{A}\right)}
\left[_{1}F_{1}\left(1,\frac{b+A+r}{A};\lambda t\right)
-1\right]\,.
\end{split}
\end{equation*}
By interchanging $\lambda$ with $\mu$, $b$ with $r$, and $U$ with $D$ 
in (\ref{Prob1-polya}), we immediately obtain (\ref{Prob1-neg-polya}). 
\end{proof}
\par
For the absolutely continuous component of the probability law of $S_t$ we have the following results.
\begin{theorem}\label{polya-absol}
Under the assumptions of Proposition \ref{polya-discr}, for all $t>0$ and $-vt < x < ct$, we have
\begin{equation}
 f(x,t\,|\,c)=\frac{\xi(\tau_*,t)\eta(\tau_*,t)}{t-\tau_*}
 \label{density1}
\end{equation}
and 
\begin{equation}
 b(x,t\,|\,c)=
 \frac{r\lambda(\lambda\tau_*)^{b/A-1}{\rm e}^{-\lambda \tau_{*}}}
 {(c+v)A\Gamma\left(\frac{b+A}{A}\right)}
 \overline{F}_{D_1}(t-\tau_*)
 \left[_{1}F_{1}\left(1,\frac{b+A+r}{A};\lambda\tau_{*}\right)-1\right]
 +\frac{\xi(\tau_*,t)\eta(\tau_*,t)}{\tau_{*}},
 \label{density2}
\end{equation}
where $\tau_{*}$ is defined in (\ref{tau}), and where 
\begin{equation}
 \xi(\tau_*,t):=\exp\left\{-\lambda \tau_{*}-\mu(t-\tau_{*})\right\}
 \frac{(\lambda\tau_*)^{b/A+1}\big[\mu(t-\tau_*)\big]^{r/A}}
 {(c+v)\,\Gamma\left(\frac{b+A}{A}\right)\Gamma\left(\frac{r}{A}\right)}\,,
 \label{xi}
\end{equation}
\begin{equation}
\begin{split}
\eta(\tau_*,t)&:=
\left[\lambda\tau_*+\mu(t-\tau_*)\right]^{-1}
\left[
_{1}F_{1}\left(1,\frac{b+A+r}{A}; \lambda\tau_*+\mu(t-\tau_*)\right)
-1-\frac{A\left[\lambda\tau_*+\mu(t-\tau_*)\right]}{b+A+r}\right] 
\\
&-(\lambda\tau_*)^{-1}
 \left[_{1}F_{1}\left(1,\frac{b+A+r}{A}; \lambda\tau_*\right)
 -1-\frac{A\lambda\tau_*}{b+A+r}\right]\,.
\label{eta}
\end{split}
\end{equation}
\end{theorem}
\begin{proof} 
Due to Eqs.\ (\ref{density-sum}), (\ref{Prob3}) and (\ref{formula-polya2}), 
and recalling (\ref{distr2}), under the given assumptions we have  
\begin{equation*}
\begin{split}
&f(x,t\,|\,c)= 
\frac{1}{c+v}\sum_{k = 2}^{+\infty}\sum_{j = 0}^{k-2}
\Prob\{N_{k-1}=j,Z_k=c\,|\,Z_0=c\}
\\
&\times
{\int^{t}_{t-\tau_{*}}}
\frac{\lambda^{b/A+j+1}(s-t+\tau_{*})^{b/A+j}}{\Gamma\left(\frac{b}{A}+j+1\right)}
{\rm e}^{-\lambda(s-t+\tau_{*})}
\frac{\mu^{r/A+k-j-1}(t-\tau_{*})^{r/A+k-j-2}}{\Gamma\left(\frac{r}{A}+k-j-1\right)}
{\rm e}^{-\mu(t-\tau_{*})}{\rm e}^{-\lambda(t-s)}{\rm d}s
\\
&=\frac{\xi(\tau_*,t)}{t-\tau_*}
\sum_{k = 2}^{+\infty}
\frac{1}{\left(\frac{b+A+r}{A}\right)_k}
\sum_{j = 0}^{k-2}
{k-1 \choose j}(\lambda\tau_*)^j
[\mu(t-\tau_*)]^{k-1-j}\\
&=\frac{\xi(\tau_*,t)}{t-\tau_*}
\sum_{k = 2}^{+\infty}
\frac{(1)_k}{\left(\frac{b+A+r}{A}\right)_k}\frac{1}{k!}
\left\{
\left[\lambda\tau_*+\mu(t-\tau_*)\right]^{k-1}-(\lambda\tau_*)^{k-1}
\right\}\\
&=\frac{\xi(\tau_*,t)}{t-\tau_*}
\left\{
\left[\lambda\tau_*+\mu(t-\tau_*)\right]^{-1}
\left[
_{1}F_{1}\left(1,\frac{b+A+r}{A}; \lambda\tau_*+\mu(t-\tau_*)\right)
-1-\frac{A\left[\lambda\tau_*+\mu(t-\tau_*)\right]}{b+A+r}
\right]\right.\\
&\left.-(\lambda\tau_*)^{-1}
\left[
_{1}F_{1}\left(1,\frac{b+A+r}{A}; \lambda\tau_*\right)
-1-\frac{A\lambda\tau_*}{b+A+r}\right]
\right\}\,.
\end{split}
\end{equation*}
Similarly, Eqs.\ (\ref{density-sum}), (\ref{Prob3}) and 
(\ref{formula-polya3}) give
\begin{equation*}
\begin{split}
b(x,t\,|\,c) &=
\frac{1}{c+v}\left\{
\sum_{k = 1}^{+\infty} \Prob\{N_{k-1}=k-1, Z_k=-v\,|\,Z_0=c\}
\frac{\lambda^{b/A+k}\tau_{*}^{b/A+k-1} {\rm e}^{-\lambda \tau_{*}}}
{\Gamma\left(\frac{b}{A}+k\right)}
\overline{F}_{D_1}(t-\tau_*)\right.
\\
&+\sum_{k = 2}^{+\infty}\sum_{j = 0}^{k-2}
\Prob\{N_{k-1}=j, Z_k=-v\,|\,Z_0=c\}
\\
& \left. \times \frac{\lambda^{b/A+j+1}\tau_{*}^{b/A+j}\mu^{r/A+k-j-1} 
{\rm e}^{-\lambda \tau_{*}-\mu(t-\tau_{*})}}
{\Gamma\left(\frac{b}{A}+j+1\right)\Gamma\left(\frac{r}{A}+k-j-1\right)}
\int_{\tau_*}^t (s-\tau_{*})^{r/A+k-j-2}
{\rm d}s\right\}\\
&=
\frac{r(\lambda\tau_*)^{b/A}}{A\Gamma\left(\frac{b+A}{A}\right)\tau_*(c+v)}
{\rm e}^{-\lambda \tau_{*}}\overline{F}_{D_1}(t-\tau_*)
\sum_{k = 1}^{+\infty}
\frac{(1)_k}{\left(\frac{b+A+r}{A}\right)_k}\frac{(\lambda\tau_{*})^{k}}{k!}
\\
&+\frac{\xi(\tau_*,t)}{\tau_{*}}
\sum_{k = 2}^{+\infty}
\frac{1}{\left(\frac{b+A+r}{A}\right)_k}
\sum_{j = 0}^{k-2}
{k-1 \choose j}(\lambda\tau_{*})^{j}[\mu(t-\tau_{*})]^{k-1-j}\\
&=
\frac{r(\lambda\tau_*)^{b/A}}{A\Gamma\left(\frac{b+A}{A}\right)\tau_*(c+v)}
{\rm e}^{-\lambda \tau_{*}}\overline{F}_{D_1}(t-\tau_*)
\left[_{1}F_{1}\left(1,\frac{b+A+r}{A};\lambda\tau_{*}\right)-1\right]
\\
&+\frac{\xi(\tau_*,t)}{\tau_{*}}
\sum_{k = 2}^{+\infty}
\frac{(1)_k}{\left(\frac{b+A+r}{A}\right)_k}\frac{1}{k!}
\left[
[\lambda\tau_{*}+\mu(t-\tau_{*})]^{k-1}-(\lambda\tau_*)^{k-1}
\right]\\
&=
\frac{r(\lambda\tau_*)^{b/A}}{A\Gamma\left(\frac{b+A}{A}\right)\tau_*(c+v)}
{\rm e}^{-\lambda \tau_{*}}\overline{F}_{D_1}(t-\tau_*)
\left[_{1}F_{1}\left(1,\frac{b+A+r}{A};\lambda\tau_{*}\right)-1\right] 
+\frac{\xi(\tau_*,t)}{\tau_{*}}\\
&\times\left\{
[\lambda\tau_{*}+\mu(t-\tau_{*})]^{-1}
\left[_{1}F_{1}\left(1,\frac{b+A+r}{A};\lambda\tau_{*}+\mu(t-\tau_{*})\right)
-1-\frac{A[\lambda\tau_{*}+\mu(t-\tau_{*})]}{b+A+r}
\right]
\right.\\
&\left. 
-(\lambda\tau_*)^{-1}
\left[_{1}F_{1}\left(1,\frac{b+A+r}{A};\lambda\tau_{*}\right)
-1-\frac{A(\lambda\tau_{*})}{b+A+r}
\right]
\right\}\,.
\end{split}
\end{equation*}
The proof is thus complete.
\end{proof}
\begin{corollary}\label{cor:pPol}
Under the assumptions of Theorem \ref{polya-absol}, for all $t>0$ and $-vt <x< ct$, 
the probability density of $S_t$ is given by Eq.\ (\ref{eq:defdensp}), where
\begin{equation}\label{p-pos}
p(x,t\,|\,c)=
\frac{r\lambda(\lambda\tau_*)^{b/A-1}{\rm e}^{-\lambda \tau_{*}}}
{(c+v)A\Gamma\left(\frac{b+A}{A}\right)}
\overline{F}_{D_1}(t-\tau_*)
\left[_{1}F_{1}\left(1,\frac{b+A+r}{A};\lambda\tau_{*}\right)-1\right]
+\frac{\xi(\tau_*,t)\eta(\tau_*,t)\,t}{\tau_{*}(t-\tau_*)}\,,
\end{equation}
\begin{equation}\label{p-neg}
\begin{split}
p(x,t\,|\,-v)&=
\frac{b\mu\big[\mu(t-\tau_*)\big]^{r/A-1}{\rm e}^{-\mu(t- \tau_{*})}}
{(c+v)A\Gamma\left(\frac{r+A}{A}\right)}
\overline{F}_{U_1}(\tau_*)
\left[_{1}F_{1}\left(1,\frac{b+A+r}{A};\mu(t-\tau_{*})\right)-1\right]\\
&+
\frac{\widetilde{\xi}(\tau_*,t)\widetilde{\eta}(\tau_*,t)\,t}
{\tau_*(t-\tau_{*})}\,,
\end{split}
\end{equation}
where $\tau_{*}$,  $\xi$ and $\eta$,  are defined respectively in 
Eqs.\ (\ref{tau}), (\ref{xi}) and (\ref{eta}), and where
\begin{equation*} 
\widetilde{\xi}(\tau_*,t):=\exp\left\{-\lambda \tau_{*}-\mu(t-\tau_{*})\right\}
\frac{(\lambda\tau_*)^{b/A}\big[\mu(t-\tau_*)\big]^{r/A+1}}
{(c+v)\Gamma\left(\frac{r+A}{A}\right)\Gamma\left(\frac{b}{A}\right)}\,,
\end{equation*}
\begin{equation*} 
\begin{split}
\widetilde{\eta}(\tau_*,t)&:=
\left[\lambda\tau_*+\mu(t-\tau_*)\right]^{-1}
\left[
_{1}F_{1}\left(1,\frac{b+A+r}{A}; \lambda\tau_*+\mu(t-\tau_*)\right)
-1-\frac{A\left[\lambda\tau_*+\mu(t-\tau_*)\right]}{b+A+r}\right] \\
&-[\mu(t-\tau_*)]^{-1}
\left[_{1}F_{1}\left(1,\frac{b+A+r}{A}; \mu(t-\tau_*)\right)
-1-\frac{A\mu(t-\tau_*)}{b+A+r}\right]\,.
\end{split}
\end{equation*}
\end{corollary}
\begin{proof}
Eq.\  (\ref{p-pos}) immediately follows from (\ref{densp}). From Remark \ref{remark-density}, 
by interchanging $\mu$ with $\lambda$, $(t-\tau_*)$ with $\tau_*$, $b$ with $r$, 
and $D$ with $U$ in $b(x,t\,|\,c)$, given in (\ref{density2}), we get the following density:
\begin{equation*} 
\begin{split}
f(x,t\,|\,-v)&=
\frac{b\mu\big[\mu(t-\tau_*)\big]^{r/A-1}}
{(c+v)A\Gamma\left(\frac{r+A}{A}\right)}
{\rm e}^{-\mu(t- \tau_{*})}\overline{F}_{U_1}(\tau_*)
\left[_{1}F_{1}\left(1,\frac{b+A+r}{A};\mu(t-\tau_{*})\right)-1\right]\\
&+\frac{\widetilde{\xi}(\tau_*,t)\widetilde{\eta}(\tau_*,t)}{t-\tau_{*}}.
\end{split}
\end{equation*}
Similarly, we obtain the density $b(x,t\,|\,-v)$ by interchanging $\mu$ with $\lambda$, 
$(t-\tau_*)$ with $\tau_*$, $b$ with $r$ and $D$ with $U$ in $f(x,t\,|\,c)$, given in 
(\ref{density1}), so that 
$$
 b(x,t\,|\,-v)=
 \frac{\widetilde{\xi}(\tau_*,t)\widetilde{\eta}(\tau_*,t)}{\tau_*}\,.
$$
Hence, Eq.\ (\ref{p-neg}) follows by (\ref{densp}).
\end{proof}
\par
Various plots of density $p(x,t)$ given in Corollary \ref{cor:pPol}
are shown by Figures 6, 7 and 8. 
\par
%%%%%%%%%  FIG 6  %%%%%% 
\begin{figure}[t]
\centerline{
\epsfxsize=7.cm
\epsfbox{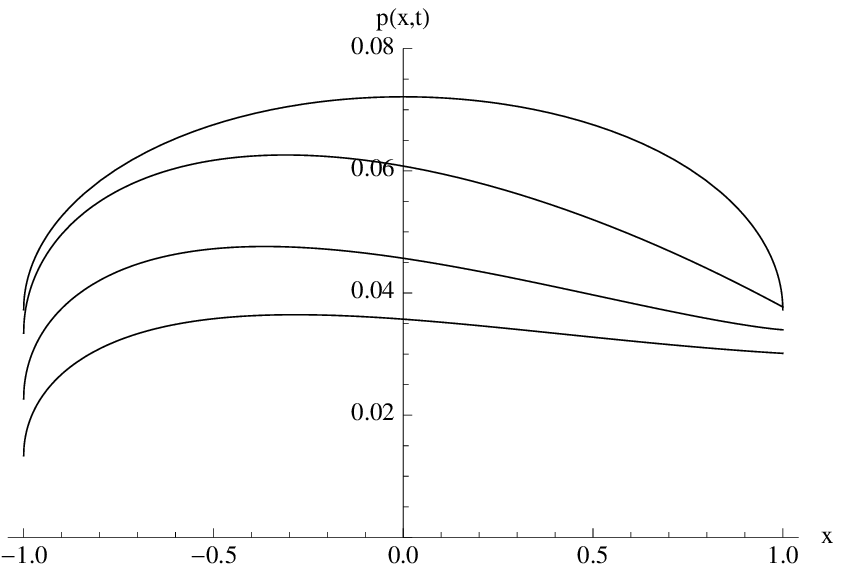}
\;
\epsfxsize=7.cm
\epsfbox{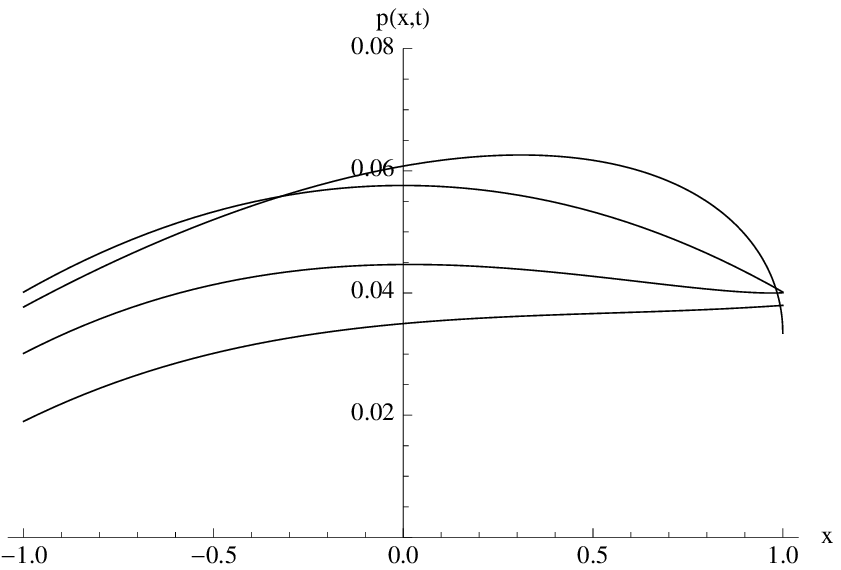}
}
\caption{
Density $p(x,t)$ for P\'olya scheme with densities (\ref{distr2}), with 
$t=1$, $c=v=1$, $\lambda=\mu=1$, $A=2$, for $b=1$ ({\em left}) and 
$b=2$ ({\em right}), and $r=1$, $2$, $3$, $4$ (from bottom to top near $x=0$).
}
\end{figure}
%%%%%%%%%%%%%%
%
%%%%%%%%%  FIG 7 %%%%%% 
\begin{figure}[t]
\centerline{
\epsfxsize=7.cm
\epsfbox{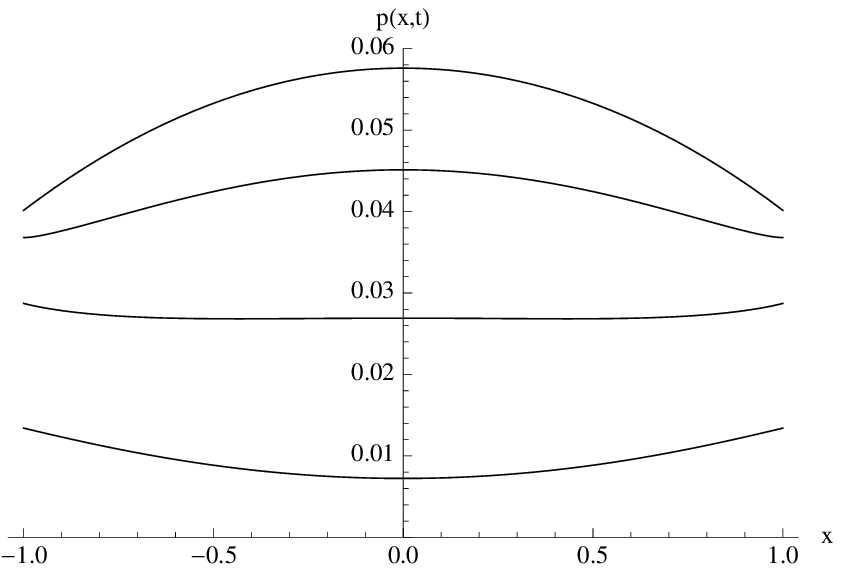}
\;
\epsfxsize=7.cm
\epsfbox{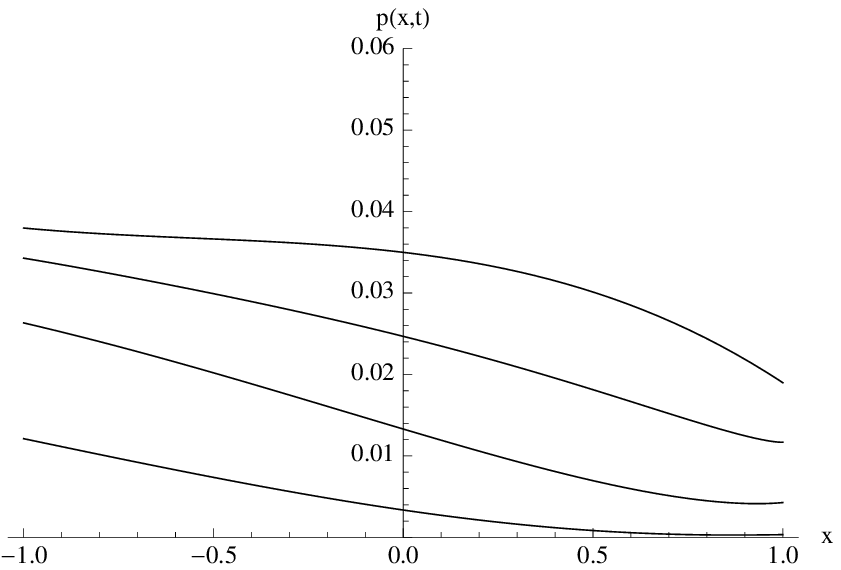}
}
\caption{
Same as Figure 6, with $r=1$ and  $A=0.4$, $0.6$, $0.8$, $1$ (from bottom to top near $x=0$).
}
\end{figure}
%%%%%%%%%%%%%%
%
%%%%%%%%%  FIG 8 %%%%%% 
\begin{figure}[t]
\centerline{
\epsfxsize=7.cm
\epsfbox{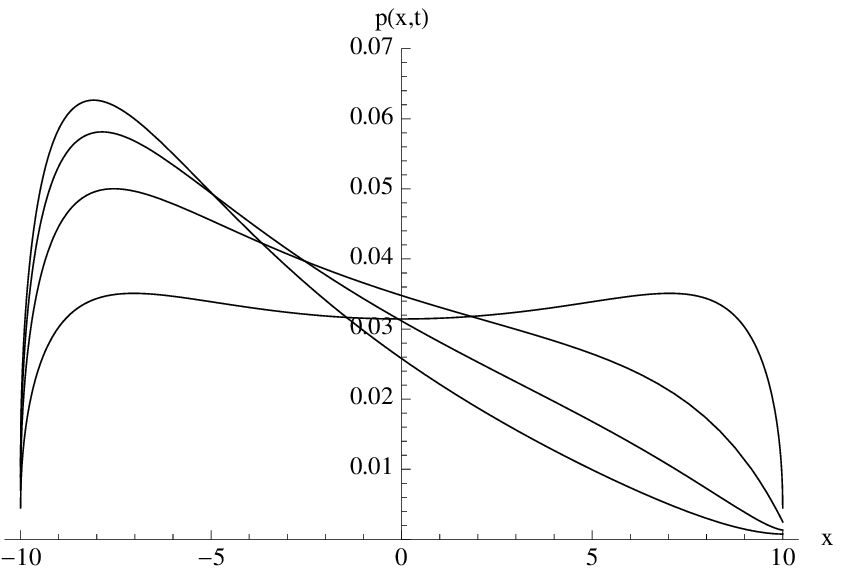}
\;
\epsfxsize=7.cm
\epsfbox{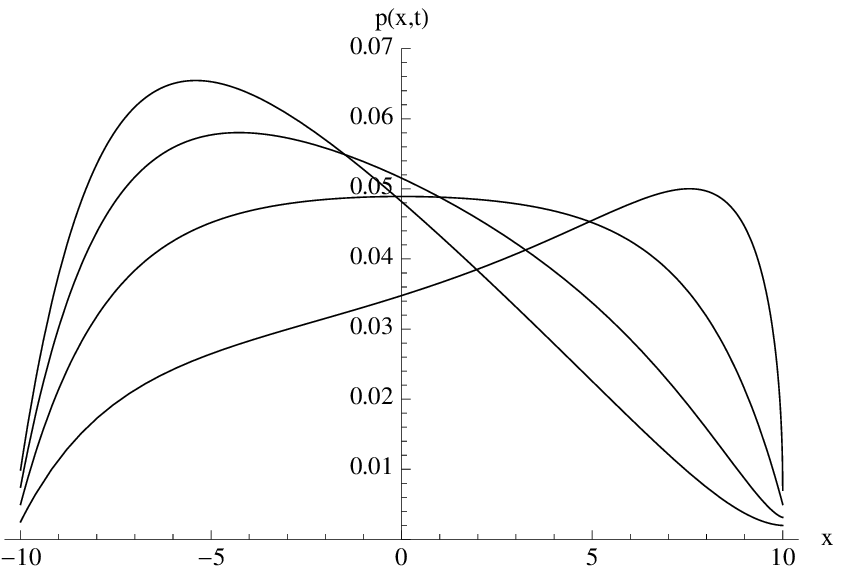}
}
\caption{
Same as Figure 6 (from bottom to top near $x=-8$), with $t=10$.
}
\end{figure}
%%%%%%%%%%%%%%
%
\begin{remark}  
\rm 
From the formulas given in Theorem \ref{polya-absol} we have  
\begin{equation*}
 \lim_{\tau_*\to 0} \frac{\xi(\tau_*,t)\eta(\tau_*,t)}{t-\tau_*}=0,
 \qquad
 \lim_{\tau_*\to t} \frac{\xi(\tau_*,t)\eta(\tau_*,t)}{t-\tau_*}=0,
\end{equation*}
$$
 \lim_{\tau_*\to 0} \frac{\xi(\tau_*,t)\eta(\tau_*,t) }{\tau_*}=0,
 \qquad 
 \lim_{\tau_*\to t} \frac{\xi(\tau_*,t)\eta(\tau_*,t) }{\tau_*}=0.
$$
Hence, some calculations allow to study the behaviour of the density $p(x,t)$ 
in proximity of the endpoints $-vt,\,ct$: 
\begin{equation*}
\lim_{x\uparrow -vt} p(x,t)=
\frac{br\mu(\mu t)^{r/A-1}{\rm e}^{-\mu t}}
{(b+r)(c+v)A\Gamma(r/A+1)}
\left[{_1}F_1\left(1,\frac{b+A+r}{A};\mu t\right)-1\right]
\end{equation*}
and
\begin{equation*}
\lim_{x\uparrow ct} p(x,t)=
\frac{br\lambda(\lambda t)^{b/A-1}{\rm e}^{-\lambda t}}
{(b+r)(c+v)A\Gamma(b/A+1)}
\left[{_1}F_1\left(1,\frac{b+A+r}{A};\lambda t\right)-1\right]\,.
\end{equation*}
\end{remark}
\par
Note that in this case a stationary distribution does not exist. Indeed, 
$$
 \lim_{t\to +\infty}\Prob\{S_t=ct\}=0,\qquad
 \lim_{t\to +\infty}\Prob\{S_t=-vt\}=0,\qquad
 \lim_{t\to +\infty} p(x,t)=0,
$$
since (see Equation 13.1.4 of \cite{Abram1994}), as $|z|\to +\infty$, 
$$
_{1}F_1(a,b;z)=\frac{\Gamma(b)}{\Gamma(a)}\, {\rm e}^z z^{a-b}[1+O(|z|^{-1})]
\qquad\hbox{for } {\rm Re}(z)>0\,.
$$ 
\par
Finally, the mean velocity subordinated to the positive initial
velocity can be expressed in the following form.
\begin{proposition}\label{th:mediaV-gamma}
Under the assumptions of Theorem \ref{polya-absol}, for all $t>0$ we have
\begin{equation}
\E [V_{t}\,|\,V_{0}=c]=c\,\overline{F}_{U_{1}}(t)+ 
c\pi_A\, \sum_{k = 1}^{+\infty}\phi_k(t\,|\,c)
+\,(-v)(1-\pi_A)\,\sum_{k = 1}^{+\infty}\psi_k(t\,|\,c)
\label{mediaV-gamma}
\end{equation}
with
\begin{equation}\label{phi-gamma}
\begin{split}
\phi_k(t\,|\,c)&=\sum_{j = 0}^{k-1}\Prob\{N_{k-1}=j\,|\,Z_0=c\}
\\
&\times \left[G(r/A+k-j-1,\mu, b/A+j+1,\lambda;t)-
G(r/A+k-j-1, \mu, b/A+j+2,\lambda;t)\right]\,,\\
\psi_k(t\,|\,c)&=\sum_{j = 0}^{k-1}\Prob\{N_{k-1}=j\,|\,Z_0=c\}
\\
&\times \left[G(b/A+j+1,\lambda, r/A+k-j-1,\mu;t)-
G(b/A+j+1, \lambda, r/A+k-j,\mu;t)\right]\,,
\end{split}
\end{equation}
where $\Prob\{N_{k-1}=j\,|\,Z_0=c\}$ is given by (\ref{formula-polya1})
and the function $G$ is defined in Appendix 
\ref{formulas}, by Eq.\ (\ref{funcG}).
\end{proposition}
\begin{proof} 
Eqs.\ (\ref{mediaV-gamma}) and (\ref{phi-gamma}) 
easily follow from Proposition \ref{th-mediaV}, Remark \ref{remarkFT}
and Eq.\ (\ref{funcG}) since, for $k\geq 1$, $U_{k+1}$ and $D_{k+1}$ are
exponential with parameter $\lambda$ and $\mu$, respectively, and
$U^{(k)}$ and $D^{(k)}$ are Gamma-distributed with parameters
$(b/A+k,\lambda)$ and $(r/A+k,\mu)$ respectively.
\end{proof}
%
%%%%%%%%%%%%%%%%%%%%%%%%%%%%%%
\subsection{A suitable extension}\label{sect:reinforced}
%%%%%%%%%%%%%%%%%%%%%%%%%%%%%%
In this section we propose a suitable extension  
of  the model based on  P\'olya scheme  
by replacing the classical P\'olya urn with the following 
randomly reinforced urn (cf.\ Aletti {\em et al.} \cite{AlMaSe09}, 
Berti {\em et al.} \cite{BeCrPrRi11}, \cite{BePrRi04}, Crimaldi \cite{Cr09}).  
An urn contains $b>0$ black balls and $r>0$
red balls. At each epoch $n\geq 1$, a ball is drawn and then replaced
together with a random number of balls of the same color. Say that
$B_n\geq 0$ black balls or $R_n\geq 0$ red balls are added to the urn
according to whether $X_n=1$ or $X_n=0$, where $X_n$ is the indicator
function of the event $\{$black ball at time $n\}$.  We assume, for
each $n\geq 1$, 
\begin{equation*}
(B_n,R_n)\mbox{ independent of }\,
\bigl(X_1,B_1,R_1,X_2,\ldots,B_{n-1},R_{n-1},X_n\bigr).
\end{equation*}
According to this new model, formula (\ref{eq:Vn}) must be replaced by
\begin{equation*} 
\begin{split}
\Prob\{Z_{n}=c\,|\,{\cal G}_n\}&=\Prob\{X_{n+1}=1\,|\,{\cal G}_n\}=
\frac{b+\sum_{k=1}^nB_kX_k}{b+r+\sum_{k=1}^n\bigl(B_kX_k+R_k(1-X_k)\bigr)},
\\
\Prob\{Z_{n}=-v\,|\,{\cal G}_n\}&=\Prob\{X_{n+1}=0\,|\,{\cal G}_n\}=
\frac{r+\sum_{k=1}^n R_k(1-X_k)}{b+r+\sum_{k=1}^n\bigl(B_kX_k+R_k(1-X_k)\bigr)}, 
 \end{split}
\end{equation*}
where
\begin{equation*}
\mathcal{G}_0=\{\emptyset,\Omega\},
\qquad
\mathcal{G}_n=\sigma\bigl(X_1,B_1,R_1,\ldots,X_n,B_n,R_n\bigr).
\end{equation*}
It has been proven that, if $B_n=R_n$ for each $n$ (cf.\ Berti et
al.\ \cite{BePrRi04} and Crimaldi \cite{Cr09}) or, more generally,
if\footnote{The notation $\stackrel{d}=$ denotes equality in distribution} 
$B_n\stackrel{d}=R_n$ for each $n$ (cf.\ Berti et al.\ \cite{BeCrPrRi11}, p.\ 538), 
then $(X_n)_{n\geq 1}$ is a sequence
of conditionally identically distributed random variables (which can be 
not exchangeable) and so we have, for $n\geq 2$, 
$$
\Prob\{X_{n}=1\,|\,X_1=1\}=\E\left[\frac{b+B_1}{b+B_1+r}\right],
\qquad 
\Prob\{X_{n}=0\,|\,X_1=1\}=\E\left[\frac{r}{b+B_1+r}\right]\,. 
$$
Using the above equalities, instead of (\ref{legge-X}), we can obtain 
a formula for the mean velocity similar to (\ref{mediaV}). 
However, in this case the computation of the conditional probability 
distribution of $N_{k-1}$ needed in Eq.\ (\ref{FTk|Z0}) is not an easy task, 
and possibly will be the object of future investigations. 
%%%%%%%%%%%%%%%%%%%%%%%%%
\appendix
%%%%%%%%%%%%%%%%%%%%%%%%%
\section{The conditional probability distribution of $N_{k-1}$}\label{appendix-N}
%%%%%%%%%%%%%%%%%%%%%%%%%
%===========================
\subsection{Bernoulli scheme $(A=0)$}
%===========================
Using the independence of $\{X_n; n\geq 1\}$ and setting $p$ as in
(\ref{eq:def_p}), we have for $k\geq 1$ and $0\leq j\leq k-1$
$$
\Prob\{N_{k-1}=j\,|\,Z_0=c\}=
\Prob\left\{\sum_{h=2}^k X_h=j\,|\,X_1=1\right\}
={k-1 \choose j}p^{j}(1-p)^{k-1-j}
$$
and so
\begin{equation}
\Prob\{N_{k-1}=j,\,Z_k=c\,|\,Z_0=c\}={k-1 \choose j}p^{j+1}(1-p)^{k-1-j}\,,
\label{formula-bern2}
\end{equation}
which simply becomes $\Prob\{Z_1=c\,|\,Z_0=c\}=p$ for $k=1$ and $j=0$. 
Similarly, we have
\begin{equation}
\Prob\{N_{k-1}=j,\,Z_k=-v\,|\,Z_0=c\}={k-1 \choose j}p^{j}(1-p)^{k-j}\,.
\label{formula-bern3}
\end{equation}
Conditioning on $Z_0=-v$, we get
\begin{equation}
\Prob\{N_{k-1}=k-1-j,\,Z_k=c\,|\,Z_0=-v\}
={k-1 \choose j}p^{k-j}(1-p)^j
\label{formula-bern2-neg}
\end{equation}
and
\begin{equation}
 \Prob\{N_{k-1}=k-1-j,\,Z_k=-v\,|\,Z_0=-v\}
 ={k-1 \choose j} p^{k-1-j}(1-p)^{j+1}\,.
 \label{formula-bern3-neg}
\end{equation}
%
%===========================
\subsection{P\'olya scheme $(A>0)$}
%===========================
Using the exchangeability of $\{X_n; n\geq 1\}$, we have for $k\geq 1$
and $0\leq j\leq k-1$
\begin{equation}
\begin{split}
\Prob\{N_{k-1}=j\,|\,Z_0=c\}&=
%\Prob\{\sum_{h=2}^k X_h=j\,|\,X_1=1\} \\
%&=
{k-1 \choose j}
\frac
{\Gamma\left(\frac{b+A+r}{A}\right)}
{\Gamma\left(\frac{b+A}{A}\right)\Gamma\left(\frac{r}{A}\right)}
\,
\frac
{\Gamma\left(j+\frac{b+A}{A}\right)\Gamma\left(k-1-j+\frac{r}{A}\right)}
{\Gamma\left(k-1+\frac{b+A+r}{A}\right)}\,, \\
&={k-1 \choose j}
 \left(\frac{b+A}{A}\right)_j \left(\frac{r}{A}\right)_{k-j-1} 
 \left[\left(\frac{b+A+r}{A}\right)_{k-1}\right]^{-1},
\end{split}
\label{formula-polya1}
\end{equation}
where $(\alpha)_0=1$ and $(\alpha)_j=\alpha(\alpha+1)\cdot\ldots\cdot
(\alpha+j-1)=\Gamma(\alpha+j)/\Gamma(\alpha)$ (the ascending factorial).  
For $k\geq 2$, the above formula corresponds to the quantity
$\Prob\{\sum_{h=2}^k X_h=j\,|\,X_1=1\}$, i.e.\ to the probability of
obtaining $j$ black balls in $k-1$ drawings in a P\'olya urn scheme
starting from $r$ red balls and $b+A$ black balls. Hence we get
\begin{equation}
\begin{split}
\Prob\{N_{k-1}=j,\,Z_k=c\,|\,Z_0=c\}&=
\Prob\{X_{k+1}=1\,|\,N_{k-1}=j,\,X_1=1\}\cdot
\Prob\{N_{k-1}=j\,|\,Z_0=c\}\\
&=\frac{b+A(j+1)}{b+r+Ak}\Prob\{N_{k-1}=j\,|\,Z_0=c\}\\
&={k-1 \choose j}
\frac
{\Gamma\left(\frac{b+A+r}{A}\right)\Gamma\left(j+1+\frac{b+A}{A}\right)
\Gamma\left(k-1-j+\frac{r}{A}\right)}
{\Gamma\left(\frac{b+A}{A}\right)\Gamma\left(\frac{r}{A}\right)
\Gamma\left(k+\frac{b+A+r}{A}\right)}
\\
&={k-1 \choose j}\left(\frac{b+A}{A}\right)_{j+1}
\left(\frac{r}{A}\right)_{k-j-1}
\left[\left(\frac{b+A+r}{A}\right)_{k}\right]^{-1}\,,
\end{split}
\label{formula-polya2}
\end{equation}
which simply becomes $\Prob\{Z_1=c\,|\,Z_0=c\}=\frac{b+A}{b+r+A}$ 
for $k=1$ and  $j=0$. Similarly, we have
\begin{equation}
\begin{split}
\Prob\{N_{k-1}=j,\,Z_k=-v\,|\,Z_0=c\}&=
\Prob\{X_{k+1}=0\,|\,N_{k-1}=j,\,X_1=1\}\cdot
\Prob\{N_{k-1}=j\,|\,Z_0=c\}\\
&=\frac{r+A(k-j-1)}{b+r+Ak}\Prob\{N_{k-1}=j\,|\,Z_0=c\}
\\
&=
{k-1 \choose j}
\frac
{\Gamma\left(\frac{b+A+r}{A}\right)\Gamma\left(j+\frac{b+A}{A}\right)
\Gamma\left(k-j+\frac{r}{A}\right)}
{\Gamma\left(\frac{b+A}{A}\right)\Gamma\left(\frac{r}{A}\right)
\Gamma\left(k+\frac{b+A+r}{A}\right)}\\
&={k-1 \choose j}\left(\frac{b+A}{A}\right)_{j}
\left(\frac{r}{A}\right)_{k-j}
\left[\left(\frac{b+A+r}{A}\right)_{k}\right]^{-1}\,.
\end{split}
\label{formula-polya3}
\end{equation}
Conditioning on $Z_0=-v$, we get
\begin{equation}
\Prob\{N_{k-1}=k-1-j,\,Z_k=c\,|\,Z_0=-v\}
={k-1 \choose j}\left(\frac{r+A}{A}\right)_{\! j}
\left(\frac{b}{A}\right)_{\! k-j}
\bigg[\left(\frac{b+A+r}{A}\right)_{\! k}\bigg]^{-1}
\label{formula-polya2-neg}
\end{equation}
and
\begin{equation}
\Prob\{N_{k-1}=k-1-j,\,Z_k=-v\,|\,Z_0=-v\}
={k-1 \choose j}\left(\frac{r+A}{A}\right)_{\! j+1}
\left(\frac{b}{A}\right)_{\! k-j-1}
\bigg[\left(\frac{b+A+r}{A}\right)_{\! k}\bigg]^{-1}\,.
\label{formula-polya3-neg}
\end{equation}
%
%%%%%%%%%%%%%%%%%%%%%%%
\section{Auxiliary functions}\label{formulas}
%%%%%%%%%%%%%%%%%%%%%%%
For the reader's convenience, now we provide some formulas used in  the proofs.  
\par
The confluent hypergeometric function (or Kummer's function) used
throughout the paper is so defined (see \cite{Abram1994} for details)
\begin{equation}\label{hypergeom-func}
%\begin{split}
_{1}F_{1}(u,v;z)=\sum_{k=0}^{\infty}\frac{(u)_k}{(v)_k}\frac{z^k}{k!}\,.
%\quad\mbox{(confluent hypergeom. func. or Kummer's func.),}
%\\
%_{2}F_{1}(u,v,w;z)&=\sum_{k=0}^{\infty}
% \frac{(u)_k(v)_k}{(w)_k}\frac{z^k}{k!}\quad
%\mbox{(Gauss hypergeom. func.).}
%\end{split}
\end{equation}
We start recalling the relation
\begin{equation}\label{hypergeom-eq}
_{1}F_1(a,b;z)={\rm e}^z\, {_{1}}F_1(b-a,b;-z)\quad \hbox{or, equivalently,}\quad 
_{1}F_1(a,b;-z)={\rm e}^{-z}\, {_{1}}F_1(b-a,b;z)\,. 
\end{equation}
An useful integration formula is Equation 3.383.1 of \cite{GrRy2007}: 
\begin{equation}\label{integral-1}
\int_0^t y^{\nu-1}(t-y)^{\mu-1} {\rm e}^{\beta y}{\rm d}y=
\frac{\Gamma(\mu)\Gamma(\nu)}{\Gamma(\mu+\nu)} t^{\mu+\nu-1} 
{_{1}}F_1(\nu,\mu+\nu;\beta t)
\end{equation}
when ${\rm Re}(\mu)>0$ and ${\rm Re}(\nu)>0$.
\par 
Let $X$ and $Y$ be independent gamma-distributed random variables 
with parameters $(\alpha,\mu)$ and $(\beta,\lambda)$, respectively.  
By (\ref{hypergeom-eq}) and the integration formula (\ref{integral-1}), we find
\begin{equation}\label{cum-distr-func-gamma}
F_X(t)=\frac{(\mu t)^{\alpha}}{\Gamma(\alpha+1)}
\,{\rm e}^{-\mu t}\,
{_{1}}F_{1}(1,\alpha+1;\mu t)\,,\qquad\mu>0,\, t\geq 0
\end{equation}
and so, for $\alpha,\,\beta,\,\mu,\,\lambda>0$ and $t\geq 0$, 
\begin{equation}\label{funcG}
\begin{split}
G(\alpha,\mu,\beta,\lambda;t):= \,& \Prob(X+Y\leq t)=\int_0^t F_X(t-y)f_Y(y){\rm d}y
\\
=\,&
\mu^\alpha\lambda^\beta t^{\alpha+\beta} {\rm e}^{-\mu t}
\sum_{h=0}^{+\infty} \frac{(\mu t)^h}{\Gamma(\alpha+\beta+h+1)}\,
  {_{1}}F_{1}(\beta,\alpha+\beta+h+1;(\mu-\lambda)t)\,.
\end{split}
\end{equation}
Note that an alternative expression of $G$ was given in Moschopoulos \cite{Mo1985} 
as series of suitable integrals. 
Analogously  to (\ref{funcG}), by (\ref{integral-1}) we can express the function 
\begin{equation}
 H\big(\alpha,\beta;t\big)
  := \int_0^t (t-y)\,_{1}F_1\big(1,2;\alpha(t-y)\big) 
{\rm e}^{-\alpha(t-y)}{\rm e}^{-\beta y}\,{\rm d}y\,,
 \qquad \alpha,\beta\in\mathbb{R},\;\; t\geq 0 
 \label{funcH}
\end{equation}
in terms of the hypergeometric function as
$$
 H\big(\alpha,\beta;t\big)=\left\{
 \begin{array}{ll}
 \displaystyle\frac{t}{\alpha} 
\left[{}_1F_1(1,2;-\beta t)-{\rm e}^{-\alpha t} {}_1F_1(1,2;(\alpha-\beta) t)\right], 
 & 
 \alpha\neq 0,\; \beta\in{\mathbb R}
 \\[3mm]
 \displaystyle\frac{t}{\beta}\left[ {\rm e}^{-\beta t} -\, {}_1F_1(1,2;-\beta t)\right],
 &
 \alpha=0, \;\beta\neq 0
 \\[3mm]
  \displaystyle\frac{t^2}{2}\,, 
   &
 \alpha=0, \;\beta= 0\,,
 \end{array} 
 \right.
$$ 
where  we recall that ${}_1F_1(1,2;z)$ reduces to $1$ when $z=0$ and to $({\rm e}^z-1)/z$
when $z\neq 0$.
%--------------------------------------------------------------------
\subsection*{Acknowledgements.}
%--------------------------------------------------------------------
The authors thank an anonymous referee for useful remarks that improved the paper.
\par
This work was partially supported by MIUR (PRIN 2008) -- within the Projects 
``Mathematical models and computation methods for information processing 
and transmission in neuronal systems subject to stochastic dynamics'' and 
``Bayesian methods: theoretical advancements and new applications" -- and  
by CNR PNR project ``CRISIS Lab''. 
%--------------------------------------------------------------------

\end{document}